\documentclass[12pt,a4paper,reqno,oneside,nosumlimits]{amsart}
\usepackage[mathcal]{eucal}
\usepackage{amssymb}
\usepackage[nice]{nicefrac}
\usepackage{hyperref}
\usepackage{color}
\usepackage[normalem]{ulem}

\definecolor{orange}{rgb}{1.0, 0.55, 0.0}

\theoremstyle{definition}
\newtheorem{definition}{Definition}[section]

\newtheorem{remark}[definition]{Remark}

\newtheorem{beispiel}[definition]{Example}

\newenvironment{example}{\begin{beispiel}%
  \pushQED{\qed}}%
  {\popQED\end{beispiel}}

\theoremstyle{plain}
\newtheorem{theorem}[definition]{Theorem}
\newtheorem{lemma}[definition]{Lemma}
\newtheorem{proposition}[definition]{Proposition}
\newtheorem{corollary}[definition]{Corollary}

\usepackage{vmargin}

\setpapersize{A4}
\setmargins{3cm}       % margen izquierdo
{1.5cm}                        % margen superior
{14cm}                      % anchura del texto
{23.42cm}                    % altura del texto
{10pt}                           % altura de los encabezados
{1cm}                           % espacio entre el texto y los encabezados
{0pt}                             % altura del pie de página
{2cm}                           % espacio entre el texto y el pie de página

\numberwithin{equation}{section}

\DeclareMathOperator{\hdim}{hdim}
\DeclareMathOperator{\hspec}{hspec}

\DeclareMathOperator{\Aut}{Aut}
\DeclareMathOperator{\Inn}{Inn}

\newcommand{\F}{\mathbb{F}}
\newcommand{\Z}{\mathbb{Z}}
\newcommand{\Q}{\mathbb{Q}}
\newcommand{\N}{\mathbb{N}}
\newcommand{\R}{\mathbb{R}}

\newcommand{\aug}{\mathfrak{a}}

\begin{document}

\title[The lower $p$-series of analytic pro\nobreakdash-$p$
groups]{The lower $p$-series of analytic pro\nobreakdash-$p$ groups \\
  and Hausdorff dimension}

\author[I. de las Heras]{Iker de las Heras} 
\address{Iker de las Heras: Department of Mathematics, University of the Basque Country UPV/EHU, 48940 Leioa, Spain}
\email{iker.delasheras@ehu.eus}

\author[B. Klopsch]{Benjamin Klopsch} \address{Benjamin Klopsch:
  Heinrich-Heine-Universit\"at D\"usseldorf,
  Mathematisch-Naturwissenschafltiche Fakult\"at, Mathematisches Institut} \email{klopsch@math.uni-duesseldorf.de}

\author[A. Thillaisundaram]{Anitha Thillaisundaram} 
\address{Anitha Thillaisundaram: Centre for Mathematical Sciences, Lund University, 223 62 Lund, Sweden}
\email{anitha.thillaisundaram@math.lu.se}

%\date{\today}

\thanks{\textit{Funding acknowledgements.} The first author was  supported by the Spanish Government, grant MTM2017-86802-P, partly  with FEDER funds, and by the Basque Government, grant IT974-16; he  received funding from the European Union's Horizon 2021 research and  innovation programme under the Marie Sklodowska-Curie grant  agreement, project 101067088. The second author received support  from the Knut and Alice Wallenberg Foundation.  The third author  acknowledges support from EPSRC, grant EP/T005068/1 and from the  Folke Lann\'{e}r's Fund.  The research was partially conducted in  the framework of the DFG-funded research training group GRK 2240: Algebro-Geometric Methods in Algebra, Arithmetic and Topology.}
    
\keywords{Lower $p$-series, $p$-adic analytic groups, Hausdorff  dimension, Hausdorff spectrum}

\subjclass[2020]{Primary 20E18; Secondary 22E20, 28A78}

%%%%%

\begin{abstract} 
  Let $G$ be a $p$-adic analytic pro\nobreakdash-$p$ group of
  dimension~$d$, with lower $p$-series
  $\mathcal{L} \colon P_i(G), \,i \in \N$.  We produce an approximate
  series which descends regularly in strata and whose terms deviate
  from $P_i(G)$ in a uniformly bounded way.  This brings to light a
  new set of rational invariants
  $\xi_1, \ldots, \xi_d \in [\nicefrac{1}{d},1]$, canonically
  associated to~$G$, such that
  \[
    \sup_{i \in \N} \; \Bigl\lvert \log_p \lvert G : P_i(G) \rvert -
    \Big( \sum_{k=1}^d \xi_k \Big) (i-1) \Bigr\rvert < \infty
  \]
  and such that the Hausdorff dimensions $\hdim_G^\mathcal{L}(H)$ of
  closed subgroups $H \subseteq_\mathrm{c} G$ with
  respect to $\mathcal{L}$ are confined to the set
  \[
  \Big\{ \Big(\sum_{k=1}^d \varepsilon_k \xi_k \Big) /
  \Big( \sum_{k=1}^d \xi_k \Big) \mid \varepsilon_1, \ldots,
  \varepsilon_d \in \{0,1\} \Big\}.
  \]
  In particular, the Hausdorff spectrum $\hspec^\mathcal{L}(G)$ is
  discrete and consists of at most $2^d$ rational numbers.
\end{abstract}

\maketitle

%%%%%

\section{Introduction}

 Throughout $p$ denotes a prime, and $\Z_p$ is the ring
of $p$-adic integers.  We are interested in the lower $p$-series of
$p$-adic analytic pro\nobreakdash-$p$ groups and questions about
Hausdorff dimension with respect to the lower $p$-series. 

\smallskip

Let $G$ be a finitely generated pro\nobreakdash-$p$ group.  The
\emph{lower $p$-series} of $G$ -- sometimes called the lower
$p$-central series -- is defined recursively as follows:
\begin{align*}
  P_1(G) = G,  %
  & \quad \text{and} \quad  P_i(G) = \overline{P_{i-1}(G)^p
    \, [P_{i-1}(G),G]} \quad
    \text{for $i\in \N$ with $i \ge 2$,}  
\end{align*}
where we adopt the notation from \cite[Sec.~1.2]{DDMS99}. It is the
fastest descending central series of $G$ with elementary abelian
sections and it consists of open normal subgroups forming a base of
neighbourhoods for the element~$1$.  The lower $p$-series plays a
significant role in the study of pro\nobreakdash-$p$ groups.  For
instance, $\bigoplus_{i=1}^\infty P_i(G)/P_{i+1}(G)$ naturally carries
the structure of an $\F_p$-Lie algebra and can frequently be used to
analyse the structure of $G$ or its finite quotients; compare with
\cite[Chap.~VIII]{HuBl82}, \cite[Chap.~II]{La65} and
\cite[Sec.~5]{La85}.

In order to understand the lower $p$-series of $p$-adic analytic
pro\nobreakdash-$p$ groups, it is useful to consider an analogous
series for $\Z_p$-lattices furnished with a group action.  Let $L$ be
a $\Z_p$-lattice equipped with a (continuous) right $G$-action, in
other words, let $L$ be a right $\Z_pG$-module that is finitely
generated and torsion-free as a $\Z_p$-module.  The two-sided ideal
\[
  \aug_G = \sum_{g\in G} (g-1) \Z_pG +
  p\Z_pG \,\trianglelefteq\, \Z_pG
\]
is the kernel of the natural epimorphism from $\Z_pG$ onto the finite
field $\F_p$ that sends each group element to~$1$; it is convenient to
refer to $\aug_G$, somewhat informally, as the \emph{$p$-augmentation
  ideal} of the group ring $\Z_pG$.  The powers $\aug_G^{\, i}$,
$i \in \N_0$, of the $p$\nobreakdash-augmentation ideal intersect in
$\{0\}$ and induce a topology on $\Z_pG$ which gives rise to the
completed group ring $\Z_p[\![G]\!]$; see \cite[Sec.~7.1]{DDMS99}.
The descending series of open $\Z_pG$-submodules
\begin{equation} \label{equ:def-lower-p-of-L} \lambda_i(L) =
  L . \aug_G^{\, i} \quad \text{for $i \in \N_0$,}
\end{equation}
is the fastest descending $G$-central series of $L$ with elementary
abelian sections.  Accordingly we call the series
\eqref{equ:def-lower-p-of-L} the \emph{lower $p$-series} of the
$\Z_pG$-module $L$, taking note of the notational shift in the index
in comparison to the lower $p$\nobreakdash-series of a group.  As $G$
is a pro\nobreakdash-$p$ group, it acts unipotently on every principal
congruence quotient $L/p^j L$, $j \in \N$, and this implies
that $\bigcap_{i \in \N_0} \lambda_i(L) = \{0\}$.

In general, a \emph{filtration series} of $L$, regarded as a
$\Z_p$-lattice, is a descending chain
$\mathcal{S} \colon L = L_0 \supseteq L_1 \supseteq
  \ldots$ of open $\Z_p$-sublattices
$L_i \subseteq_\mathrm{o} L$ such that
$\bigcap_{i \in \N_0} L_i = \{0\}$.  For $c \in \N_0$, we say that two
filtration series
$\mathcal{S} \colon L = L_0 \supseteq L_1 \supseteq
  \ldots$ and
$\mathcal{S}^* \colon L = L_0^* \supseteq L_1^*
\supseteq \ldots$ are
\emph{$c$-equivalent} if for all $i \in \N$,
\[
  p^c L_i \subseteq L_i^* \quad \text{and} \quad p^c L_i^* \subseteq
  L_i.
\]
We say that $\mathcal{S}$ and $\mathcal{S}^*$ are \emph{equivalent} if
they are $c$-equivalent for some $c \in \N_0$.  This is but
one of several ways to establish the concept or a closely
related one.  For instance, if
$\log_p \lvert L_i+L_i^* : L_i \cap L_i^* \rvert \le c$ for all
$i \in \N$, then $\mathcal{S}$ and $\mathcal{S}^*$ are
$c$-equivalent; conversely, if $\mathcal{S}$ and $\mathcal{S}^*$ are
$c$-equivalent, then
$\log_p \lvert L_i+L_i^* : L_i \cap L_i^* \rvert \le 2c$ for all
$i \in \N$.  Thus $\mathcal{S}$ and $\mathcal{S}^*$ are
equivalent if and only if $\lvert L_i+L_i^* : L_i \cap L_i^* \rvert$
is uniformly bounded as $i$ ranges over all positive integers.

Our aim is to show that the lower $p$-series of $L$ is equivalent to a
filtration series that is descending very regularly, in a sense that
we make precise now.  Let $d = \dim_{\Z_p} L$ denote the dimension of
the $\Z_p$-lattice $L$, and let $x_1, \ldots, x_d$ be a $\Z_p$-basis
for $L$.  We say that a filtration series
$\mathcal{S} \colon L = L_0 \supseteq L_1
  \supseteq \ldots$ of $L$ has a (vertical)
  \emph{stratification} given by the \emph{frame} $(x_1, \ldots, x_d)$
  if there are positive real parameters $\xi_1, \ldots, \xi_d$ with
  $\xi_1 \le \ldots \le \xi_d$ such that $\mathcal{S}$ is equivalent
  to the filtration series
\[
  \mathcal{S}^* \colon L_i^* = \bigoplus_{k=1}^d p^{\lfloor i
    \xi_k\rfloor} \Z_p \, x_k;
\]
in this situation, $(\xi_1,\ldots,\xi_d)$ is uniquely determined by
the frame and we refer to this $d$-tuple as the \emph{growth rate} of
the stratification.  Informally, we call $\xi_1,\ldots,\xi_d$ the
`growth rates' of the stratification.  We say that the stratification
$\mathcal{S}^*$ is \emph{given by $G$-invariant sublattices} if all
terms $L_i^*$ are $G$-invariant.

\begin{theorem} \label{thm:stratification-exists} Let $G$ be a
  finitely generated pro\nobreakdash-$p$ group, and let $L$ be a
  $\Z_p$-lattice equipped with a $G$-action.  Then the lower
  $p$-series of the $\Z_pG$-module $L$ admits a stratification, and
  any two stratifications have the same growth rate.

  Moreover, if $d = \dim_{\Z_p} L$ and if $(x_1, \ldots, x_d)$ is a
  frame for a stratification of the lower $p$-series of $L$ with
  growth rate $(\xi_1, \ldots, \xi_d)$, then
  \begin{enumerate}
  \item $\nicefrac{1}{d} \le \xi_1\le \ldots \le \xi_d \le 1$ are
    rational numbers,
  \item for every $e \in \{1,\ldots, d-1\}$ with $\xi_e < \xi_{e+1}$
    the sublattice $\bigoplus_{k=1}^e \Z_p x_k$ is
    $G$\nobreakdash-invariant, and
  \item the entire stratification is given by $G$-invariant
    sublattices.
  \end{enumerate}
\end{theorem}

The proof of Theorem~\ref{thm:stratification-exists} is
  relatively straightforward in the special case where
  $\Q_p \otimes_{\Z_p} L$ is a completely reducible $\Q_p G$-module;
  see Remark~\ref{rem:completely-reducible}.  But the general setting
  requires additional care.  In essence, our analysis gives rise to an
  inductive procedure for computing the growth rate parameters; see
  Remark~\ref{rem:how-to-determine-xis}.

By means of Lie-theoretic considerations,
  Theorem~\ref{thm:stratification-exists} yields an analogous
description of the lower $p$-series of a $p$\nobreakdash-adic analytic
pro\nobreakdash-$p$ group.  We recall that, for every $p$-adic
analytic pro\nobreakdash-$p$ group $G$, there exists $j \in \N$ such
that for every $i \in \N$ with $i \ge j$ the group $P_i(G)$ is
uniformly powerful; see \cite[Prop.~3.9 and Thm.~4.2]{DDMS99}.
Moreover, every minimal generating system $x_1, \ldots, x_d$ of a
uniformly powerful pro\nobreakdash-$p$ group $U$ yields a product
decomposition of
$U = \overline{\langle x_1 \rangle} \cdots \overline{\langle x_d
  \rangle}$ into $d$ procyclic subgroups, which in turn gives a system
of $p$-adic coordinates $U \to \Z_p^{\, d}$,
$x_1^{\, a_1} \cdots x_d^{\, a_d} \mapsto (a_1,\ldots,a_d)$ of `the
second kind'; see \cite[Thm.~4.9]{DDMS99}.  The dimension of $G$,
denoted by $\dim(G)$, coincides with the rank of any uniformly
powerful open subgroup of~$G$.

\begin{theorem} \label{thm:p-series-of-p-adic-group} Let $G$ be a
  $p$-adic analytic pro\nobreakdash-$p$ group of dimension~$d$.  Then
  there exists $j \in \N$ such that
  \begin{enumerate}
  \item the group $U=P_j(G) \trianglelefteq_\mathrm{o} G$ is uniformly
    powerful and
  \item there is a minimal generating system
    $x_1,\ldots, x_d$ for $U$ and there are rational
    numbers $\xi_1,\ldots,\xi_d$ with
    $\nicefrac{1}{d} \le \xi_1 \le \ldots \le \xi_d \le 1$ and
    $c \in\N_0$ such that the following
    holds: for
    all $i \in \N$ with $i \ge \max\{j+2cd,6cd\}$,
   the following products  of  procyclic
      subgroups
    \begin{align*}
      \check{U}_i %
      & = \overline{\langle x_1^{\, p^{\lfloor (i -
        j)\xi_1\rfloor+c}}\rangle} \,\cdots\, \overline{\langle
        x_d^{\, p^{\lfloor (i-j) \xi_d\rfloor+c}} \rangle} \\
      \hat{U}_i %
      &= \overline{\langle x_1^{\, p^{\lfloor
        (i-j)\xi_1\rfloor-c}} \rangle} \,\cdots\,
        \overline{\langle x_d^{\, p^{\lfloor (i-j)\xi_d\rfloor -c}}\rangle}
    \end{align*}
    form open normal subgroups
    $\check{U}_i, \hat{U}_i \trianglelefteq_\mathrm{o} G$ such that
    \[
      \check{U}_i \subseteq P_i(G) \subseteq \hat{U}_i
    \]
    and the factor group $\hat{U}_i / \check{U}_i$ is homocyclic
    abelian of exponent~$p^{2c}$.
  \end{enumerate}
\end{theorem}

\begin{corollary} \label{cor:index-of-p-series-term} In the set-up of
  Theorem~\ref{thm:p-series-of-p-adic-group} and with the additional
  notation $\sigma = \sum_{k=1}^d \xi_k \in [1,d]$, the index of
  $P_i(G)$ in $G$, for $i \ge \max\{j+2cd,6cd\}$, satisfies
  \[
    \Bigl\lvert \log_p \lvert G : P_i(G) \rvert - \big( \log_p \lvert
    G : P_j(G) \rvert + \lfloor (i-j) \sigma \rfloor \big) \Bigr\rvert \le
    d(c+1).
  \]
\end{corollary}

\begin{corollary} \label{cor:xi-unique-for-groups} Let $G$ be a
  $p$-adic analytic pro\nobreakdash-$p$ group of dimension~$d$.  Then
  any two approximations of the lower $p$-series as described in
  Theorem~\ref{thm:p-series-of-p-adic-group} involve the same rational
  growth parameters
  $\nicefrac{1}{d} \le \xi_1 \le \ldots \le \xi_d \le 1$.
\end{corollary}

The last corollary underlines that we have uncovered a new set of
invariants which capture certain structural properties of $p$-adic
analytic pro\nobreakdash-$p$ groups.

\smallskip

One of our motivations for pinning down the lower $p$-series of
$p$-adic analytic pro\nobreakdash-$p$ groups was to settle some basic,
but annoyingly outstanding questions about the Hausdorff spectra of
such groups.  The study and application of Hausdorff dimension in
profinite groups is an ongoing affair; for instance, see
\cite{KlThZu19} for an overview and
\cite{KlTh19,GaGaKl20,GoZo21,HeKl22,HeTh22a,HeTh22b} for some recent
developments.  Let $G$ be an infinite finitely generated
pro\nobreakdash-$p$ group and let
$\mathcal{S} \colon G_i,\, i \in \N_0$, be a \emph{filtration series}
of~$G$, that is, a descending chain
$G = G_0 \supseteq G_1 \supseteq \ldots$ of open normal
subgroups $G_i \trianglelefteq_\mathrm{o} G$ such that
$\bigcap_{i \in \N_0} G_i = 1$.  In accordance with
\cite{Ab94,BaSh97}, the \emph{Hausdorff dimension} of a closed
subgroup $H \subseteq_\mathrm{c} G$, with respect
to~$\mathcal{S}$, admits an algebraic interpretation as a logarithmic
density:
\begin{equation*}
  \hdim_G^\mathcal{S}(H) =
  \varliminf_{i\to \infty} \frac{\log_p \lvert HG_i : G_i
    \rvert}{\log_p \lvert G : G_i \rvert},
\end{equation*}
and the \emph{Hausdorff spectrum} of $G$, with respect to
$\mathcal{S}$,
\[
  \hspec^\mathcal{S}(G) = \{ \hdim_G^\mathcal{S}(H) \mid  H
  \subseteq_\mathrm{c} G \text{ a subgroup} \} \subseteq [0,1]
\]
provides a simple `snapshot' of the typically rather complicated array
of all subgroups in~$G$.  We say that $H$ has \emph{strong Hausdorff
  dimension} in $G$, with respect to $\mathcal{S}$, if
\[
  \hdim^{\mathcal{S}}_G(H) = \lim_{i \to \infty} \frac{\log_p \lvert H
    G_i : G_i \rvert}{\log_p \lvert G : G_i \rvert}
\]
is given by a proper limit.

Even for well-behaved groups, such as $p$-adic analytic
pro\nobreakdash-$p$ groups~$G$, not only the Hausdorff dimension of
individual subgroups, but also qualitative features of the Hausdorff
spectrum of~$G$ are known to be sensitive to the choice
of~$\mathcal{S}$; see~\cite[Thm.~1.3]{KlThZu19}.  In our setting,
where $G$ is a finitely generated pro\nobreakdash-$p$ group, the
group-theoretically relevant choices of $\mathcal{S}$ are primarily:
the $p$-power series~$\mathcal{P}$, the iterated $p$-power
series~$\mathcal{P}^*$, the Frattini series~$\mathcal{F}$, the
dimension subgroup (or Zassenhaus) series~$\mathcal{D}$ and, last but
not least, the lower
$p$-series~$\mathcal{L} \colon P_i(G),\, i \in \N$.  It was shown
in~\cite[Prop.~1.5]{KlThZu19} that, if $G$ is a $p$-adic analytic
pro\nobreakdash-$p$ group, then the Hausdorff dimensions of
a subgroup $H \subseteq_\mathrm{c} G$ with respect to
the series $\mathcal{P}, \mathcal{F}, \mathcal{D}$
coincide. Furthermore,~\cite[Thm.~1.1]{BaSh97} yields,
for $\mathcal{S}$ any one of these series
$\mathcal{P}, \mathcal{F}, \mathcal{D}$, that the spectrum
\[
  \hspec^\mathcal{S}(G) = \Big\{ \nicefrac{\displaystyle
    \dim(H)}{\displaystyle \dim(G)} \mid  H
    \subseteq_\mathrm{c} G \text{ a subgroup} \Big\}
\]
is discrete and that all closed subgroups have strong Hausdorff
dimension in $G$ with respect to $\mathcal{S}$.  These results extend
easily to Hausdorff dimensions with respect to the series
$\mathcal{P}^*$.

It was seen in~\cite{KlThZu19} that, contrary to indications
in~\cite{BaSh97}, the Hausdorff spectra of $p$-adic analytic
pro\nobreakdash-$p$ groups $G$ with respect to the lower $p$-series
$\mathcal{L}$ follow a different pattern.  Concretely,
\cite[Example~4.1]{KlThZu19} provides a family of $p$\nobreakdash-adic
analytic pro\nobreakdash-$p$ groups $G$ such that
$\lvert \hspec^\mathcal{L}(G) \rvert / \dim(G)$ is unbounded as
$\dim(G) \to \infty$.  It remained open whether
$\hspec^\mathcal{L}(G)$ could even be infinite for some $p$-adic
analytic groups~$G$.  Theorem~\ref{thm:stratification-exists} allows
us to establish that $\hspec^\mathcal{L}(G)$ is in fact
finite for every $p$-adic analytic pro\nobreakdash-$p$ group $G$ and
to put in place further restrictions.

\begin{theorem} \label{thm:spectrum-p-adic} Let $G$ be an
    infinite $p$-adic analytic pro\nobreakdash-$p$ group and let
  $d = \dim(G)$.  Then there exist
  $\xi_1,\ldots,\xi_d \in [\nicefrac{1}{d},1] \cap \Q$ such that
  \[
    \hspec^{\mathcal{L}}(G) \subseteq \left\{\frac{\varepsilon_1\xi_1+
        \ldots + \varepsilon_d\xi_d}{\xi_1 + \ldots + \xi_d} \mid
      \varepsilon_1, \ldots, \varepsilon_d \in \{0,1\} \right\}.
  \]
  
  In particular, this shows that $\lvert \hspec^{\mathcal{L}}(G) \rvert \le 2^d$.  The parameters $\xi_1, \ldots, \xi_d$ are the same as those appearing in Theorem~\ref{thm:p-series-of-p-adic-group}, and all closed subgroups have strong Hausdorff dimension in~$G$ with respect to~$\mathcal{L}$.
\end{theorem}

It is a natural problem to improve, if possible, upon the upper bound
for $\lvert \hspec^{\mathcal{L}}(G) \rvert$ as a function of
$d = \dim(G)$.  In Example~\ref{ex:max-value} we manufacture a family
of $p$-adic analytic pro\nobreakdash-$p$ groups $G_m$,
$m \in \N$, such that
\[
  \frac{\log_p \bigl\lvert \hspec^\mathcal{L}(G_m)
    \bigr\rvert}{\sqrt{\dim(G_m)}} \to \infty \qquad \text{as
    $m \to \infty$.}
\]

Theorem~\ref{thm:spectrum-p-adic} triggers more refined questions about Hausdorff spectra with respect to the lower $p$-series.  We address some of these in a separate paper~\cite{HeKlTh26}. For instance, one outstanding problem is to understand what renders the Hausdorff spectrum of a finitely generated pro-$p$ group finite. Using Lie-theoretic tools we show that non-abelian finite direct products of free pro-$p$ groups have full Hausdorff spectrum $[0,1]$ with respect to the lower $p$-series, in analogy to~\cite{GaGaKl20}.

\medskip

\noindent \textit{Notation.}  Generally, any new notation is explained
when needed, sometimes implicitly.  For convenience we collect basic
conventions here.

We denote by $\N$ the set of positive integers and by $\N_0$ the set
of non-negative integers.  Throughout, $p$ denotes a prime, and $\Z_p$
is the ring of $p$-adic integers (or simply its additive group).  We
write $\varliminf a_i = \varliminf_{i \to \infty} a_i$ for the lower
limit (limes inferior) of a sequence $(a_i)_{i \in \N}$ in
$\R \cup \{ \pm \infty \}$.  Intervals of real numbers are written as
$[a,b]$, $(a,b]$ et cetera.
 
The group-theoretic notation is mostly standard and in line, for
instance, with its use in~\cite{DDMS99}.  Tacitly, subgroups of
profinite groups are generally understood to be closed subgroups.  In
some places, we emphasise that we are taking the topological closure
of a set $X$ by writing~$\overline{X}$.

\medskip

\noindent \textit{Organisation.}  In
Section~\ref{sec:stratifications}, we study stratifications of the
lower \texorpdfstring{$p$}{p}-series and establish
Theorems~\ref{thm:stratification-exists} and
\ref{thm:p-series-of-p-adic-group}, as well as their corollaries.  In
Section~\ref{sec:hspec} we turn to Hausdorff spectra with respect to
the lower \texorpdfstring{$p$}{p}-series: using the results from the
previous section, we prove Theorem~\ref{thm:spectrum-p-adic}.

\medskip

\noindent \textbf{Acknowledgements.} 
We thank both referees for very useful comments and suggestions.

%%%%%

\section{Stratifications of the lower
  \texorpdfstring{$p$}{p}-series} \label{sec:stratifications}

Let $G$ be a profinite group, and let $L$ be a $\Z_p$-lattice equipped
with a right $G$-action.  We are interested in situations, where the
lower $p$-series $\lambda_i(L)$, $i \in \N_0$, of $L$, as defined in
\eqref{equ:def-lower-p-of-L}, satisfies
$\bigcap_{i \in \N_0} \lambda_i(L) = \{0\}$ and hence gives a
filtration series.  This is the case if and only if $G$ acts
unipotently on every principal congruence quotient $L/p^j L$,
$j \in \N$, which in turn is true if and only if the action of $G$ on
the $\Z_p$-lattice $L \cong \Z_p^{\, d}$ factors through a
pro\nobreakdash-$p$ subgroup of $\mathsf{GL}_d(\Z_p)$.  Since every
closed subgroup of $\mathsf{GL}_d(\Z_p)$ is topologically finitely
generated, we may restrict without any true loss of generality to the
situation that $G$ itself is a finitely generated pro\nobreakdash-$p$
group.  A theorem of Serre implies that every abstract action of the
group $G$ on the group $L$ is continuous; see~\cite[Sec.~1]{DDMS99}.
Thus we arrive at the standard set-up which we use henceforth: $L$ is
a $\Z_p$-lattice equipped with a right $G$-action, for a finitely
generated pro\nobreakdash-$p$ group~$G$.

\smallskip

We start with a basic observation about growth rates of 
stratifications.

\begin{lemma} \label{lem:same-growth-rates} Let
$\mathcal{S} \colon L = L_0 \supseteq L_1 \supseteq
    \ldots$ be a filtration series of a $\Z_p$-lattice $L$ which
  admits a stratification.  Then any two stratifications of
  $\mathcal{S}$ have the same growth rate.
\end{lemma}

\begin{proof}
  Let $d = \dim_{\Z_p} L$ and, for $i \in \N$, let
  $p^{m_{i,1}}$, \ldots, $p^{m_{i,d}}$ with
  $0 \le m_{i,1} \le \ldots \le m_{i,d}$ be the elementary divisors of
  the finite $\Z_p$-module $L/L_i$ so that
  \[
    L/L_i \cong \Z_p / p^{m_{i,1}}\Z_p \,\oplus \ldots \oplus\,
    \Z_p / p^{m_{i,d}}\Z_p.
  \]
  
  Let $(x_1,\ldots,x_d)$ be the frame for a stratification of
  $\mathcal{S}$, with growth rate $(\xi_1, \ldots, \xi_d)$, say. Then
  there is a constant $c \in \N$ such that $\mathcal{S}$ is
  $c$-equivalent to the filtration series
  $\bigoplus_{k=1}^d p^{\lfloor i \xi_k \rfloor} \Z_p x_k$,
  $i \in \N_0$.  Recall that $\xi_1 \le \ldots \le \xi_d$ and choose
  $j \in \N$ such that
  $\lfloor j \xi_e \rfloor + c < \lfloor j \xi_{e+1} \rfloor - c$ for
  all $e \in \{1, \ldots, d-1\}$ with $\xi_e <
  \xi_{e+1}$.  By definition of stratification, this
  implies that for all $i \in \N$ with $i \ge j$,
  \[
    \lvert m_{i,k} - \lfloor i \xi_k \rfloor \rvert \le c \qquad
    \text{for $1 \le k \le d$.}
  \]
  Thus each $\xi_k = \lim_{i \to \infty} m_{i,k}/i$ is uniquely
  determined by~$\mathcal{S}$.
\end{proof}

Next we recall and adapt the following definitions from~\cite{Kl03}:
the \emph{lower} and \emph{upper level} of an open $\Z_pG$-submodule
$M$ in~$L$ are
\[
  \ell_L(M) = \min \{k\in \N_0 \mid p^kL \subseteq M \} \;\text{and}\;
  u_L(M) = \max \{k\in \N_0 \mid M \subseteq p^kL\},
\]
and the \emph{rigidity} of~$L$ is
\[
  r(L) = \sup \{\ell_L(M) - u_L(M) \mid \text{$M$ is an open
    $\Z_pG$-submodule of $L$} \}.
\]
For the trivial lattice $L = \{0\}$ the formalism yields $r(L)= -\infty$.

The following auxiliary result is a variation of
\cite[Prop.~4.3]{Kl03} and can be proved in the same way, using
compactness.

\begin{proposition}
  \label{pro:rigid}
  Let $G$ be a finitely generated pro\nobreakdash-$p$ group, and let
  $L$ be a $\Z_p$-lattice equipped with a right $G$-action.  Then
  $r(L)$ is finite if and only if $\Q_p\otimes_{\Z_p} L$ is a simple
  $\Q_pG$-module.
\end{proposition}

In conjunction with Lemma~\ref{lem:same-growth-rates}, the next
proposition already establishes a special case of
Theorem~\ref{thm:stratification-exists}.

\begin{proposition}
  \label{pro:simple}
  Let $G$ be a finitely generated pro\nobreakdash-$p$ group, and let
  $L$ be a $\Z_p$-lattice equipped with a right $G$-action.  Suppose
  that $\Q_p\otimes_{\Z_p} L$ is a simple $\Q_pG$-module.  Then the
  lower $p$-series of the $\Z_pG$-module $L$ admits a stratification
  with constant growth rate $(\xi,\ldots,\xi)$, where
  $\xi \in [\nicefrac{1}{d},1] \cap \Q$ for $d = \dim_{\Z_p}L$.
\end{proposition}

\begin{proof}
  Denote by $L_i = \lambda_i(L)$, $i \in \N_0$, the terms of the lower
  $p$-series of $L$, and for $i \in \N_0$ put $u_i = u_L(L_i)$.  By
  Proposition~\ref{pro:rigid} the $\Z_pG$-module $L$ has finite
  rigidity $r =r(L)$.  Consequently, the lower $p$-series of $L$ is
  $r$-equivalent to the filtration series
 $\mathcal{S} \colon L \supseteq p^{u_1}L \supseteq
    p^{u_2}L \supseteq \ldots$, and it suffices to show that
  $\mathcal{S}$ is equivalent to a series of the form
  $p^{\lfloor i \xi \rfloor} L$, $i \in \N_0$, for suitable
  $\xi \in [\nicefrac{1}{d},1] \cap \Q$.
    
  Write $L = \bigoplus_{k=1}^d \Z_p x_k$ and choose, as a
  reference group, a homogeneous abelian group of exponent $p^r$ on
  $d$ generators,
  $R =\langle z_1, \ldots, z_d\rangle \cong
  C_{p^r}\times\overset{d}{\ldots}\times C_{p^r}$.  For each $i\in\N$
  consider the epimorphism
  \begin{equation*}
    \varphi_i \colon p^{u_i}L   \to p^{u_i}L / p^{u_i+r}L \cong  R 
  \end{equation*}
  defined by the dictum $p^{u_i} x_k \mapsto z_k$ for $1 \le k \le d$.
  For each $i \in \N$, we have $\ell_L(L_i) \le u_i + r$, hence
  $p^{u_i+r} L \subseteq L_i \subseteq p^{u_i} L$
  by the definition of $u_i$ and~$r$, and we consider
  $R_i = L_i \varphi_i$.

  Since $R$ is finite, there exist $j, m \in \N$ such that
  $R_{j+m} = R_j$, and therefore $L_{j+m} = p^n L_j $ for
  $n = u_{j+m} - u_j \ge 1$.  Furthermore $p^m L_j \subseteq L_{j+m}$
  yields $n \le m$.  We see that
  $L_{i+m} = L_{j+m} . \aug_{\Z_pG}^{\, i-j} = p^n L_j . 
  \aug_{\Z_pG}^{\, i-j} = p^n L_i$ for all $i \in \N$
  with $i\ge j$.  In particular, this shows that $u_{i+m} = u_i + n$
  for $i \ge j$.  Put $\xi = n/m \in (0,1] \cap \Q$ and
  \[
   c = \bigl\lceil 1 + \max \big\{ \lvert u_i - i \xi \rvert \bigm\vert 0 \le i
    \le j+m-1 \big\} \bigr \rceil \in \N.
  \]
  Then
  \[
    \lvert u_i - \lfloor i \xi \rfloor \rvert \le \lvert u_i - i \xi
    \rvert + 1 \le c \qquad \text{for all $i \in \N_0$,}
  \]
  and the filtration series $\mathcal{S}$ is $c$-equivalent to
  $\mathcal{S}^* \colon p^{\lfloor i \xi \rfloor} L$, $i \in \N_0$.
  Furthermore, we notice that
  $L = L_0 \supsetneqq L_1 \supsetneqq \ldots$ yields the trivial
  lower bound $\lvert L : L_i \rvert \ge p^i$.   Combining
    these two results, we arrive at
  \[
    \xi = \lim_{i \to \infty} \frac{d \lfloor i \xi \rfloor}{di} =
    \lim_{i \to \infty} \frac{\log_p \lvert L : p^{\lfloor i \xi
        \rfloor} L \rvert}{di} = \lim_{i \to \infty} \frac{\log_p
      \lvert L : L_i \rvert}{di} \ge \nicefrac{1}{d}. \qedhere
  \] 
\end{proof}
  
The next proposition provides a key tool for proving
Theorem~\ref{thm:stratification-exists}; with its corollary it yields,
in particular, the last two assertions in the theorem.

\begin{proposition} \label{pro:stratification-restricts-to-small-xi}
  Let $G$ be a finitely generated pro\nobreakdash-$p$ group, and let $L$ be a
  $\Z_p$-lattice equipped with a right $G$-action.  Suppose
  that the lower $p$-series of $L$ admits a stratification, with the
  frame $(x_1, \ldots, x_d)$ and growth rate $(\xi_1, \ldots, \xi_d)$.

  Suppose that $e \in \{1,\ldots,d-1\}$ is such that
  $\xi_e < \xi_{e+1}$.  Then the $\Z_p$-lattice
  $M = \bigoplus_{k=1}^e \Z_p x_k$ is $G$-invariant.  Moreover
  the lower $p$-series of $M$ has a stratification with frame
  $(x_1,\ldots,x_e)$ and growth rate $(\xi_1,\ldots,\xi_e)$.
\end{proposition}

\begin{proof}
  The $\Z_p$-lattice $L$ decomposes as the direct sum of its
  sublattices $M$ and $N = \bigoplus_{k=e+1}^d \Z_p x_k$.
  Choose $c_1 \in \N$ such that the lower $p$-series
  $L_i = \lambda_i(L)$, $ i \in \N_0$, is $c_1$-equivalent to
  the series
  $\bigoplus_{k=1}^d p^{\lfloor i \xi_k \rfloor} \Z_p x_k$,
  $i \in \N_0$.  Then $\xi_{e+1} = \min \{ \xi_{e+1},
    \ldots, \xi_d \}$ implies
  \begin{equation} \label{equ:Li-cap-N-small}
    L_i \cap N  \subseteq \left( \bigoplus_{k=1}^d p^{\lfloor i
        \xi_k \rfloor - c_1} \Z_p x_k \right) \cap N \subseteq
  p^{\lfloor i \xi_{e+1} \rfloor -c_1} 
    N \quad \text{for all $i \in \N_0$,}
  \end{equation}
  where we embed $L$ into $\Q_p \otimes_{\Z_p} L$ to interpret
  negative exponents occurring for small values of~$i$.

  For a contradiction, assume that $M$ is not $G$-invariant, i.e.\ not
  a $\Z_pG$-submodule.  Then there are
    $\alpha_1, \ldots, \alpha_e \in \Z_pG$ such that
    $y = \sum_{k=1}^e x_k \alpha_k \in N \smallsetminus \{1\}$.
    Choose $c_2 \in \N$ such that $y \in N \smallsetminus p^{c_2}N$.
  For every $i \in \N$, we  deduce from
    $\xi_e = \max \{ \xi_1, \ldots, \xi_e \}$ that
  $p^{\lfloor i \xi_e \rfloor + c_1}M \subseteq L_i$ and hence
  \[
    p^{\lfloor i \xi_e \rfloor + c_1} y \in (L_i \cap N)
    \smallsetminus p^{\lfloor i \xi_e \rfloor + c_1 + c_2}
    N.
  \]
  Thus \eqref{equ:Li-cap-N-small} yields
  $\lfloor i \xi_e \rfloor + c_1 + c_2 > \lfloor i \xi_{e+1} \rfloor
  -c_1$.  Since $\xi_e < \xi_{e+1}$ this yields the desired
  contradiction for sufficiently large $i \in \N$.

  \smallskip

  It remains to prove that the lower $p$-series $M_i = \lambda_i(M)$,
  $i \in \N_0$, of $M$ has a stratification with the frame
  $(x_1,\ldots,x_e)$ and growth rate $(\xi_1,\ldots,\xi_e)$.  As
  before we embed $L$ into $\Q_p \otimes_{\Z_p} L$ to interpret any
  negative powers of~$p$, whose occurrence does not cause
    any serious problems.  We put
  \[
    b = c_1 + b_1 \qquad \text{with} \qquad b_1 = \left\lceil
      \frac{3c_1 + 3}{\xi_{e+1}-\xi_e} \right\rceil.
  \]
  From the choice of $c_1$ and from $b \ge c_1$ we deduce
    that, for all $i \in \N_0$,
  \[
    M_i \subseteq L_i \cap M \subseteq \left(
      \bigoplus_{k=1}^d p^{\lfloor i \xi_k\rfloor - c_1}
      \Z_p \, x_k \right) \cap M \subseteq p^{-b}
    \,\bigoplus_{k=1}^e p^{\lfloor i \xi_k\rfloor}
    \Z_p \, x_k.
  \]
   Thus it suffices to show that, for all $i \in \N_0$ with
  $i \ge b_1$,
  \[
   p^b \, \bigoplus_{k=1}^e p^{\lfloor i \xi_k\rfloor}
    \Z_p \, x_k  \subseteq M_i.
  \]
  
  Let $i \in \N_0$ with $i \ge b_1$.  From $p \in \aug_G$ it
  follows that
  $p^{b_1} M_{i - b_1} \subseteq M_{i-b_1} . \aug_G^{\,
    b_1} = M_i$, and it suffices to establish that
  $p^{c_1} \bigoplus_{k=1}^e p^{\lfloor i \xi_k\rfloor}
  \Z_p \, x_k \subseteq M_{i - b_1}$.  There is no harm in
  working modulo $p \big( p^{c_1} \bigoplus_{k=1}^e p^{\lfloor i \xi_k\rfloor}
  \Z_p \, x_k \big)$ and hence it is enough to show that
  \[
    \bigoplus_{k=1}^e p^{\lfloor i \xi_k\rfloor + c_1}
    \Z_p \, x_k \subseteq M_{i - b_1} +
    \bigoplus_{k=1}^e p^{\lfloor i \xi_k\rfloor + c_1 + 1}
    \Z_p \, x_k.
  \]
  Recall that $L = M \oplus N$ as a $\Z_p$-lattice and let
  $\tau \colon L \to M$ denote the projection along~$N$.  Recall
  further that
  $\bigoplus_{k=1}^e p^{\lfloor i \xi_k\rfloor + c_1}
  \Z_p \, x_k \subseteq L_i$.  Hence it is enough to show that
  \begin{equation} \label{equ:L-i-tau-small} L_i \tau \subseteq M_{i -
      b_1} + \bigoplus_{k=1}^e p^{\lfloor i \xi_k\rfloor +
      c_1 + 1} \Z_p \, x_k.
  \end{equation}
  From
  $L_{b_1} \subseteq M + p^{\lfloor b_1 \xi_{e+1} \rfloor - c_1}N$ we
  conclude that
  \[
    L_i \tau = \big( L_{b_1} . \aug_G^{\, i - b_1} \big) \tau \subseteq
    \big( \underbrace{M . \aug_G^{\, i - b_1}}_{=M_{i - b_1}} +
    p^{\lfloor b_1 \xi_{e+1} \rfloor - c_1} \underbrace{N . \aug_G^{\, i
        - b_1}}_{\subseteq L_{i - b_1}} \big) \tau,
  \]
  and hence
    $L_{i-b_1} \subseteq \bigoplus_{k=1}^d p^{\lfloor (i-b_1) \xi_k
      \rfloor - c_1}\Z_px_k$ gives
  \begin{align*}
    L_i \tau %
    & \subseteq M_{i - b_1} + \big( p^{\lfloor b_1 \xi_{e+1} \rfloor -
      c_1} L_{i - b_1} \big)\tau \\
    & \subseteq M_{i-b_1} +
      \bigoplus_{k=1}^e p^{\lfloor b_1 \xi_{e+1} \rfloor +
      \lfloor (i-b_1) \xi_k \rfloor - 2c_1} \Z_p \, x_k.
  \end{align*}
  Now our choice of $b_1$ guarantees that, for each
  $k \in \{1,\ldots,e\}$,
  \begin{multline*}
    \lfloor b_1 \xi_{e+1} \rfloor + \lfloor (i-b_1) \xi_k \rfloor -
    2c_1 %
    \ge b_1 \xi_{e+1} + (i-b_1) \xi_k - 2c_1 -  2\\
    = i\xi_k + c_1 + b_1 (\xi_{e+1}-\xi_k) -(3c_1 + 2) \ge \lfloor
    i\xi_k\rfloor + c_1 + 1,
  \end{multline*}
  and this yields \eqref{equ:L-i-tau-small}.
\end{proof}

\begin{corollary} \label{cor:stratif-by-G-inv-lattices} Let $G$ be a
  finitely generated pro\nobreakdash-$p$ group, and let $L$ be a
  $\Z_p$-lattice equipped with a right $G$-action.  Then every
  stratification of the lower $p$-series of $L$ is given by
  $G$-invariant sublattices.
\end{corollary}

\begin{proof}
  Suppose that the lower $p$-series of $L$ admits a stratification,
  with the frame $(x_1, \ldots, x_d)$ and growth rate
  $(\xi_1, \ldots, \xi_d)$.  There are
  $\vartheta_1 < \ldots < \vartheta_r$, for some
  $r \in \{1,\ldots,d\}$, such that
  $\{\xi_k \mid 1 \le k \le d \} = \{ \vartheta_s \mid 1 \le s \le r
  \}$.  For $1 \le s \le r$ let $d_s$ denote the multiplicity of
  $\vartheta_s$ in $(\xi_1, \ldots, \xi_d)$, and for $0 \le s \le r$
  put $D(s) = d_1 + \ldots + d_s$.  Then the $\Z_p$-lattice $L$
  decomposes as
  \[
    L = N_1 \oplus \ldots \oplus N_r, \quad \text{where} \quad N_s =
    \bigoplus_{k = D(s-1)+1}^{D(s)}
    \Z_p x_k \quad \text{for $1 \le s \le r$,}
  \]
  and our stratification of the lower $p$-series of $L$ is realised by
  the series
  \begin{equation} \label{equ:series-in-Ns}
    \bigoplus_{s=1}^r p^{\lfloor i \vartheta_s \rfloor} N_s,
    \quad i \in \N_0.
  \end{equation}
  For each $s \in \{1, \ldots, r\}$,
  Proposition~\ref{pro:stratification-restricts-to-small-xi} yields
  $N_s .\aug_G \subseteq N_1 \oplus \ldots \oplus N_s$, because the
  $\Z_p$-lattice $N_1 \oplus \ldots \oplus N_s$ is
  $G$-invariant.  Recalling $\vartheta_1 < \ldots < \vartheta_r$ we
  deduce that for every $i \in \N_0$,
  \begin{multline*}
    \Big( \bigoplus_{s=1}^r p^{\lfloor i \vartheta_s \rfloor}
    N_s \Big) .\aug_G = \bigoplus_{s=1}^r p^{\lfloor i
      \vartheta_s \rfloor}
    \big( N_s .\aug_G \big) \\
    \subseteq \bigoplus_{s=1}^r p^{\lfloor i \vartheta_s
      \rfloor} \big( N_1 \oplus \ldots \oplus N_s \big) =
    \bigoplus_{s=1}^r p^{\lfloor i \vartheta_s \rfloor} N_s.
  \end{multline*}
  Hence all terms of the series \eqref{equ:series-in-Ns} are
  $G$-invariant.
\end{proof}

\begin{lemma} \label{lem:commensurable-OK} Let $G$ be a finitely
  generated pro\nobreakdash-$p$ group, and let $L, \tilde L$ be
  $\Z_p$-lattices equipped with $G$-actions such that
  $\Q_p \otimes_{\Z_p} L \cong \Q_p \otimes_{\Z_p} \tilde L$ as
  $\Q_pG$-modules.  Put $d = \dim_{\Z_p} L = \dim_{\Z_p} \tilde L$.
  If the lower $p$-series of $L$ admits a stratification with growth
  rate $(\xi_1, \ldots, \xi_d)$ then so does the lower $p$-series of
  $\tilde L$.
\end{lemma}

\begin{proof}
  Without loss of generality, $L$ and $\tilde L$ are $G$-invariant
  full $\Z_p$-lattices in one and the same
  $\Q_pG$-module $V = \Q_p L = \Q_p \tilde L$
  and $d = \dim_{\Q_p}V$.  Then there exists
  $c_1 \in \N$ such that $\tilde L \subseteq p^{-c_1}L$.
  Since $p^{-c_1}L$ and $L$ are isomorphic as $\Z_pG$-modules,
  we may assume without loss of generality that
  $\tilde L \subseteq L$.

  Furthermore, there is $c_2 \in \N$ such that
  $p^{c_2} L \subseteq \tilde L$.  This implies that
  $p^{c_2} \lambda_i(L) \subseteq \lambda_i(\tilde L) \subseteq
  \lambda_i(L)$ for all $i \in \N$.  Let $(x_1,\ldots,x_d)$ be
  the frame for a stratification of the lower $p$-series of $L$, with
  growth rate $(\xi_1, \ldots, \xi_d)$.  Then there is
  $c_3 \in \N$ such that, for $i \in \N$,
  \[
    \bigoplus_{k=1}^d p^{\lfloor i \xi_k \rfloor +c_3}
    \Z_p x_k \subseteq  \lambda_i(L)  \subseteq
    \bigoplus_{k=1}^d p^{\lfloor i \xi_k \rfloor -c_3}
    \Z_p x_k.
  \]
  Below we produce a $\Z_p$-basis $y_1,\ldots,y_d$ for
  $\tilde L$ and $c_4 \in \N$ such that
  \begin{equation} \label{equ:bounds-y-frame}
    \bigoplus_{k=1}^d p^{\lfloor i \xi_k \rfloor +c_4}
    \Z_p x_k \subseteq \bigoplus_{k=1}^d p^{\lfloor i
      \xi_k \rfloor} \Z_p y_k \subseteq
    \bigoplus_{k=1}^d p^{\lfloor i \xi_k \rfloor}
    \Z_p x_k.
  \end{equation}
  With $c = c_2+c_3+c_4$, this yields for all $i \in \N_0$,
  \[
    \bigoplus_{k=1}^d p^{\lfloor i \xi_k \rfloor +c}
    \Z_p y_k \subseteq \lambda_i(\tilde L) \subseteq
    \bigoplus_{k=1}^d p^{\lfloor i \xi_k \rfloor - c}
    \Z_p y_k.
  \]
  Thus the lower $p$-series of $\tilde L$ admits a stratification with
  frame $(y_1,\ldots,y_d)$ and growth rate $(\xi_1,\ldots,\xi_d)$.

  \smallskip

  It remains to establish \eqref{equ:bounds-y-frame} for suitable
  $y_1,\ldots,y_d$ and $c_4$.  For each $e \in \{1, \ldots, d\}$ we
  choose
  \[
    y_e = a_{e,1} x_1 + \ldots + a_{e,e-1} x_{e-1} + p^{b_e} x_e \in
    \tilde L,
  \]
  where $a_{e,1}, \ldots, a_{e,e-1} \in \Z_p$ denote suitable
  coefficients and $b_e \in \{0,1, \ldots, c_2 \}$, in such a way that
  $\bigoplus_{k=1}^e \Z_p y_k = \tilde L \cap
  \bigoplus_{k=1}^e \Z_p x_k$.  We observe that
  $y_1, \ldots, y_d$ do form a $\Z_p$-basis for $\tilde L$ and
  that $b_1 + \ldots + b_d = \log_p \lvert L : \tilde L \rvert$.
  Furthermore, we see that for each $i \in \N$ and each $e \in
  \{1, \ldots, d \}$,
    \[
    p^{\lfloor i \xi_e \rfloor} y_e = \left( \sum_{k=1}^{e-1} \big( p^{\lfloor i
      \xi_e \rfloor - \lfloor i \xi_k \rfloor} a_{e,k} \big) \big( p^{\lfloor i
      \xi_k \rfloor} x_k \big) \right) + p^{b_e} \big( p^{\lfloor i
    \xi_e \rfloor} x_e \big).
  \]
  This implies that
  \[
    \bigoplus_{k=1}^d p^{\lfloor i \xi_k \rfloor}
    \Z_p y_k \subseteq \bigoplus_{k=1}^d p^{\lfloor i \xi_k
      \rfloor} \Z_p x_k
  \]
  with
  \[
    \log_p \left\lvert \bigoplus_{k=1}^d p^{\lfloor i \xi_k
        \rfloor} \Z_p x_k : \bigoplus_{k=1}^d
      p^{\lfloor i \xi_k \rfloor} \Z_p y_k \right\rvert = b_1 +
    \ldots + b_d = \log_p \lvert L : \tilde L \rvert.
  \]
  Hence \eqref{equ:bounds-y-frame} holds for
  $c_4 = \log_p \lvert L : \tilde L \rvert$.
\end{proof}

We are now in a position to prove
Theorem~\ref{thm:stratification-exists}, but it is worth to comment
first on a special case.

  \begin{remark}\label{rem:completely-reducible}
    Proposition~\ref{pro:simple} and Lemma~\ref{lem:commensurable-OK}
    already yield a direct proof of
    Theorem~\ref{thm:stratification-exists} in the special case where
    $\Q_p \otimes_{\Z_p} L$ is a completely reducible $\Q_p G$-module.
    In this context we recall the following fact: if the pro-$p$ group
    $G$ is an open compact subgroup of a semisimple $p$-adic group,
    i.e.\ the group $\mathbf{G}(\Q_p)$ of $\Q_p$-rational points of a
    semisimple linear algebraic group $\mathbf{G}$ over $\Q_p$, then
    every $\Q_pG$-module $V$ with $\dim_{\Q_p}(V) < \infty$ is
    completely reducible.  Indeed, for every open subgroup $H$ of $G$,
    the $\Q_pG$-module $V$ is completely reducible if and only if it
    is completely reducible as a $\Q_pH$-module, essentially by the
    standard Clifford theorem and Maschke theorem.  Hence, by passing to
    an open subgroup of~$G$, we may assume that $G$ is uniformly
    powerful and form the $\Q_p$-Lie algebra
    $\mathcal{L}_G = \Q_p \otimes_{\Z_p} \log(G)$ associated to~$G$.
    In this set-up, it is easily seen that $\Q_p G$-submodules of $V$
    are the same as $\mathcal{L}_G$-submodules of~$V$.  Since the Lie
    algebra $\mathcal{L}_G$ is semisimple of characteristic zero,
    Weyl's theorem shows that $V$ is completely reducible.
  \end{remark}

% Proof of Theorem 1.1
\begin{proof}[Proof of Theorem~\ref{thm:stratification-exists}]
  Let $L$ be a $\Z_p$-lattice equipped with a $G$-action, for a
  finitely generated pro\nobreakdash-$p$ group $G$.  In view of
  Lemma~\ref{lem:same-growth-rates},
  Proposition~\ref{pro:stratification-restricts-to-small-xi} and
  Corollary~\ref{cor:stratif-by-G-inv-lattices}, it remains to prove
  that the lower $p$-series of the $\Z_pG$-module $L$ admits a
  stratification with rational growth rates in
    $[\nicefrac{1}{d},1]$, where $d = \dim_{\Z_p}L$.

  \smallskip
  
  We argue by induction on $d$.  Of course, the claim holds trivially
  if $L = \{0\}$, and we now suppose that $d \ge 1$.  By
    intersecting a simple $\Q_p G$-submodule of
    $\Q_p \otimes_{\Z_p} L$ with the original lattice~$L$, we obtain
    a $\Z_pG$-submodule $N$ of~$L$ such that $\Q_p \otimes_{\Z_p} N$
    is a simple $\Q_pG$-module.  Choose $n \in \N$ such that
    $p^{n+1} L \cap N \subseteq pN$.  Then $\tilde L = p^n L + N$ is a
    $G$-invariant open $\Z_p$-sublattice of~$L$, and $N$ is a direct
    summand of the $\Z_p$-lattice $\tilde L $.  By
    Lemma~\ref{lem:commensurable-OK}, there is no harm in replacing
    $L$ with $\tilde L$ and we may thus suppose that $L$ itself
    decomposes as $L = M \oplus N$ for a $\Z_p$-sublattice $M$.

  Suppose that we can arrange for $M$ to be $G$-invariant.  Then $L$
  is a direct sum $M \oplus N$ of $\Z_pG$-modules and its lower
  $p$-series decomposes accordingly:
  $\lambda_i(L) = \lambda_i(M) \oplus \lambda_i(N)$ for $i \in \N_0$.
  If both $\dim_{\Z_p} M$ and $\dim_{\Z_p} N$ are less than~$d$, our
  claim follows by induction, while in the remaining case, $L = N$ and
  $M = \{0\}$, the proof can be finished by appealing to
  Proposition~\ref{pro:simple}.   Using
    Lemma~\ref{lem:commensurable-OK}, the same argument can be applied
    if $N$ admits a $G$-invariant almost complement $M'$ in $L$, i.e.\
    a $G$-invariant $\Z_p$-sublattice $M'$ which complements $N$ to
    yield an open sublattice $M' \oplus N$ of $L$.

 Now suppose that $N$ has no $G$-invariant almost complement in $L$.
  Let us fix some additional notation.  We write
  \[
    L_i = \lambda_i(L), \quad i \in \N_0,
  \]
  for the lower $p$-series.  Put $e = \dim_{\Z_p} M$.  By
  Proposition~\ref{pro:simple} there is
  $\eta \in [\nicefrac{1}{(d-e)},1] \cap \Q$ such that the lower
  $p$-series of $N$ is equivalent to the series
  $p^{\lfloor i \eta \rfloor} N$, $i \in \N_0$.  By induction on the
  dimension, the lower $p$-series of the $\Z_pG$-module $L/N$, which
  is $(L_i+N)/N$, $i \in \N$, admits a stratification, based on a
  frame $(\mathbf{x}_1,\ldots, \mathbf{x}_e)$, with growth rates
  $\xi_1 \le \ldots \le \xi_e$ in $[\nicefrac{1}{e},1] \cap \Q$.  Set
  $\xi_0 = 0$ for notational convenience.

  We choose representatives $x_1, \ldots, x_e \in M$ for the cosets
  $\mathbf{x}_1,\ldots, \mathbf{x}_e \in L/N$, and pick
  $\widetilde e \in \{0,1,\ldots,e-1\}$ as large as possible subject to
  the conditions:
  \[
    \xi_{\widetilde e} < \xi_{\widetilde e + 1} \qquad \text{and} \qquad
    \tilde M = \bigoplus_{k=1}^{\widetilde e} \Z_p x_k
    \quad \text{is $G$-invariant.}
  \]
  In fact, we arrange that among all possible choices -- 
    for~$M$ among all possible almost complements for~$N$ in $L$ (with
    a subsequent reduction to $L = M \oplus N$), for the
  stratification based on the frame
  $(\mathbf{x}_1,\ldots, \mathbf{x}_e)$ and for the representatives
  $x_1, \ldots, x_e$ in $L$ -- we make one that renders $\widetilde e$ as
  large as possible.  In addition we choose
  $\widehat e \in \{ \widetilde e + 1, \ldots, e\}$ maximal with
  $\xi_{\widehat e} = \xi_{\widetilde e +1}$, and for short we put
  \[
    \vartheta = \xi_{\widetilde e +1} = \xi_{\widetilde e + 2} = \ldots =
    \xi_{\widehat e} \qquad \text{and} \qquad T =
    \bigoplus_{k=\widetilde e +1}^{\widehat e} \Z_p x_k.
  \]
  Our motivation here is that, by
  Proposition~\ref{pro:stratification-restricts-to-small-xi}, the
  $\Z_p$-lattice $\tilde M \oplus T \oplus N$ is
  $G$-invariant, even though $\tilde M \oplus T$ is not.

  \smallskip

    As on previous occasions we embed $L$ into $\Q_p \otimes_{\Z_p} L$
    to interpret terms involving negative powers of~$p$ and lying
    outside~$L$.  The next step is to show that there exists
    $b_0 \in \N$ such that
  \begin{equation} \label{equ:min-xi-eta}
    L_i \subseteq M \oplus
    p^{\min \{ \lfloor i \vartheta \rfloor, \lfloor i \eta \rfloor \}
      - b_0} N \qquad \text{for all $i \in \N_0$.}
  \end{equation}

  Indeed, the $\Z_p$-lattice $L$ clearly decomposes as a direct sum
  \begin{equation} \label{equ:M-K-N-decomposition}
    L = \tilde M \oplus K \oplus N
  \end{equation}
  of its sublattices $\tilde M$,
  $K = \bigoplus_{k= \widetilde e +1}^e \Z_p x_k$ and
  $N$.  Furthermore, the lattices $\tilde M$ and $N$ are
  $G$-invariant, whereas $K$ is not.

  Recall that $\aug_G \trianglelefteq \Z_pG$ denotes the
  $p$-augmentation ideal and yields the description
  $L_i = L . \aug_G^{\, i}$, $i \in \N_0$, of the lower
  $p$-series of~$L$, and similarly for~$N$.  According to our set-up, there
  exists $b_1 \in \N$ such that for all $i \in \N_0$,
  \[
    p^{b_1} \big( K . \aug_G^{\, i} \big) \subseteq p^{b_1} L_i \subseteq
    \tilde M \oplus p^{\lfloor i \vartheta \rfloor} K \oplus N \qquad
    \text{and} \qquad p^{b_1} \big( N . \aug_G^{\, i} \big) \subseteq
    p^{\lfloor i \eta \rfloor} N.
  \]
  By decomposing elements $y \in L$ according to
  \eqref{equ:M-K-N-decomposition} as
  $y = y_{\downharpoonleft \tilde M} + y_{\downharpoonleft K} +
  y_{\downharpoonleft N}$ into their $\tilde M$-, $K$- and $N$-parts
  and by tracing the `trajectory' of such parts as we descend along
  the lower $p$-series, we can make the inclusion recorded above more
  precise.  We choose to work modulo the $G$-invariant lattice
  $\tilde M$ and can thus focus on the $K$- and $N$-parts.  The
  challenge is to control the contributions of $K$-parts, because they
  can give rise to non-trivial $K$- and $N$-parts at the next step so
  that the trajectory splits.  In contrast, $N$-parts can only
  produce $N$-parts at the next step, because $N$ is $G$-invariant.
  At each stage along a trajectory we have, for any given $y \in L$,
  \begin{align*}
    y . \aug_G %
    & \in \tilde M \oplus (y . \aug_G)_{\downharpoonleft K}
      \oplus (y . \aug_G)_{\downharpoonleft N} \\
    & \subseteq \tilde M \oplus
      (y_{\downharpoonleft K}
      . \aug_G)_{\downharpoonleft K} \oplus (y_{\downharpoonleft K}
      . \aug_G)_{\downharpoonleft N} \oplus (y_{\downharpoonleft N}
      . \aug_G).
  \end{align*}
  We introduce a parameter~$j$, which indicates that we follow the
  $K$-parts for $j$ steps, then switch to $N$-parts and continue to
  follow the latter for the remaining $i-j-1$ steps.  For instance,
  for $i = 4$ and $j=2$, we would consider, for any given $y \in K$
  and $a_1, \ldots, a_4 \in \aug_G$, the sequence of elements
  \[
  y_0 = y, \quad y_1 = (y_0 . a_1)_{\downharpoonleft K}, \quad y_2 =
  (y_1 . a_2)_{\downharpoonleft K}, \quad y_3 = (y_2
  . a_3)_{\downharpoonleft N}, \quad y_4 = y_3 . a_4.
  \]
  Concretely, we deduce that for all $i \in \N_0$,
  \begin{align*}
    K . \aug_G^{\, i} %
    & \subseteq \tilde M \oplus (K . \aug_G^{\, i})_{\downharpoonleft
      K} \oplus \sum_{j=0}^{i-1} \bigl( \bigl( \bigl( (K . \aug_G^{\, j})_{\downharpoonleft
      K} \bigr) . \aug_G \bigr)_{\! \downharpoonleft N} \bigr) . \aug_G^{\, i-1-j} \\
    & \subseteq \tilde M \oplus p^{\lfloor i
      \vartheta \rfloor - b_1} K \oplus \sum_{j=0}^{i-1}
      p^{\lfloor j \vartheta \rfloor - b_1} N . \aug_G^{\, i-1-j} \\
    & \subseteq \tilde M \oplus p^{\lfloor i
      \vartheta \rfloor - b_1} K \oplus \sum_{j=0}^{i-1}
      p^{\lfloor j \vartheta \rfloor + \lfloor (i-1-j) \eta \rfloor- 2
      b_1} N  \\
    & \subseteq \tilde M \oplus p^{\lfloor i
      \vartheta \rfloor - b_1} K \oplus p^{ \min \{ \lfloor (i-1)
      \vartheta  \rfloor, \lfloor (i-1)\eta \rfloor \} - 2 b_1-2} N.
  \end{align*}
  This yields, for $i \in \N_0$,
  \begin{multline*}
    L_i = \tilde M . \aug_G^{\, i} + K . \aug_G^{\, i} + N
    . \aug_G^{\, i} \subseteq M + K . \aug_G^{\, i}
    + N . \aug_G^{\, i} \\
    \subseteq M + K . \aug_G^{\, i} + p^{\lfloor i \eta \rfloor - b_1}
    N \subseteq M \oplus p^{ \min \{ \lfloor i \vartheta \rfloor, \lfloor i
      \eta \rfloor \} - 2 b_1-3} N.
  \end{multline*}
  Thus \eqref{equ:min-xi-eta} holds for $b_0 = 2b_1+3$.%

  \smallskip
  
  Next we establish that there exist
  \begin{multline*}
    b, c \in \N, \quad n_i \in \N_0 \text{ with }
    \min \{ \lfloor i \vartheta \rfloor, \lfloor i \eta \rfloor \} - b
    \le n_i \le \lfloor i \eta \rfloor \\
    \text{and} \quad w_{i,k} \in p^{\min \{ \lfloor i \vartheta
      \rfloor, \lfloor i \eta \rfloor \} - b} N, \quad \text{for
      $i \in \N_0$ and $\widetilde e < k \le e$,}
  \end{multline*}
  such that the lower $p$-series $L_i$, $i \in \N_0$, is
  $c$-equivalent to the filtration series
  \begin{equation} \label{equ:first-approxiamtion-of-Li}
    \bigoplus_{k=1}^{\widetilde e} p^{\lfloor i \xi_k \rfloor}
    \Z_p x_k \;\oplus\; \bigoplus_{k = \widetilde e +1}^e
    \Z_p \big( p^{\lfloor i \xi_k \rfloor} x_k + w_{i,k} \big)
    \;\oplus\; p^{n_i} N, \quad i \in \N_0.
  \end{equation}
  
  \smallskip
  
  Indeed, our construction already incorporates a stratification of
  the lower $p$-series modulo $N$ and accordingly there exists
  $c_1 \in \N$ such that for all $i \in \N_0$,
  \[
    p^{c_1} \bigoplus_{k=1}^e  p^{\lfloor i \xi_k
      \rfloor} \Z_p x_k \subseteq L_i + N \quad \text{and} \quad
    p^{c_1} L_i \subseteq \bigoplus_{k=1}^e 
    p^{\lfloor i \xi_k \rfloor} \Z_p x_k \;\oplus\; N.
  \]
  Moreover, the $G$-invariant $\Z_p$-lattice
  $\tilde M = \bigoplus_{k=1}^{\widetilde e} \Z_p x_k$ embeds
  isomorphically as $(\tilde M + N)/N$ into $L/N$.  By
  Proposition~\ref{pro:stratification-restricts-to-small-xi}, applied
  to $L/N$, there exists $c_2 \in \N$ such that for all
  $i \in \N_0$,
 \[
   p^{c_2} \bigoplus_{k=1}^{\widetilde e}  p^{\lfloor
     i \xi_k \rfloor} \Z_p x_k \subseteq \lambda_i(\tilde M) \subseteq L_i.
 \]
 We put $c_3 = \max \{ c_1, c_2 \}$ and observe that, a fortiori, for
 all $i \in \N_0$,
 \[
   p^{c_3} p^{\lfloor i \xi_k \rfloor} x_k \in L_i \qquad \text{for $1
     \le k \le \widetilde e$.}
 \]
 Choose $j \in \N$ such that $L_j \subseteq p^{c_3}L$.  Then
 for all $i \in \N_0$ with $i \ge j$,
 \[
   p^{c_3} \bigoplus_{k=1}^e  p^{\lfloor i \xi_k
     \rfloor} \Z_p x_k \subseteq (L_i + N) \cap p^{c_3} L = L_i + (N \cap
   p^{c_3}L) = L_i + p^{c_3} N
 \]
 allows us to select
 $w_{i,\widetilde e + 1}, \ldots, w_{i,e} \in N$ such
 that
 \[
   p^{c_3} \big( p^{\lfloor i \xi_k \rfloor} x_k + w_{i,k} \big) \in
   L_i \qquad \text{for $\widetilde e < k \le e$.}
 \]
 We set $b = b_0 + c_3$ and deduce from \eqref{equ:min-xi-eta} that
 \[
   w_{i,k} \in p^{\min \{ \lfloor i \vartheta
      \rfloor, \lfloor i \eta \rfloor \} - b} N \qquad \text{for
      $i \in \N_0$ with $i \ge j$ and $\widetilde e < k \le e$.}
  \]
  For $i \in \N_0$ with $i < j$ we set
  $w_{i,\widetilde e +1} = \ldots = w_{i,e} = 0$; small
  values of $i$ play an insignificant role and this simple choice has
  no relevance beyond the fact that the series in
  \eqref{equ:first-approxiamtion-of-Li} is properly defined.  To
  streamline the notation, we also set
 \begin{equation} \label{equ:streamline-w-equal-0}
   w_{i,1} = \ldots = w_{i,\widetilde e} = 0 \qquad \text{for all
     $i \in \N_0$.}
 \end{equation}
 
 For $i \in \N_0$ with $i \ge j$ we deduce directly from our
 definitions that
 \[
   p^{c_3} \Big( \bigoplus_{k=1}^e \Z_p \big( p^{\lfloor i \xi_k
     \rfloor} x_k + w_{i,k} \big) \;\oplus\; (L_i \cap N) \Big)
   \subseteq L_i,
 \]
 and next we observe that
 \[
   p^{c_1} L_i \subseteq \bigoplus_{k=1}^e p^{\lfloor i \xi_k \rfloor}
   \Z_p x_k \;\oplus\; N = \bigoplus_{k=1}^e \Z_p \big( p^{\lfloor i
     \xi_k \rfloor} x_k + w_{i,k} \big) \;\oplus\; N
 \]
 implies
 \[
   p^{c_1+c_3} L_i = p^{c_3} \big( p^{c_1} L_i \big) \subseteq \Big(
   \bigoplus_{k=1}^e \Z_p \underbrace{p^{c_3} \big( p^{\lfloor i \xi_k
       \rfloor} x_k + w_{i,k} \big)}_{\in L_i} \Big) \;\oplus\; (L_i
   \cap N).
 \]
 Choosing $c_4 \in \N$ with $c_4 \ge c_1+c_3$ sufficiently large to
 cover also the finitely many terms indexed by $i < j$, we have
 established that $L_i$, $i \in \N_0$, is $c_4$-equivalent to the
 filtration series
 \[
   \bigoplus_{k=1}^e \Z_p \big( p^{\lfloor i \xi_k \rfloor} x_k +
   w_{i,k} \big) \;\oplus\; (L_i \cap N), \quad i \in \N_0.
 \]
 In view of \eqref{equ:min-xi-eta} and $b \ge b_0$, we deduce from our
 choice of $\eta$, in connection with the lower $p$-series of $N$,
 that there is $c_5 \in \N$ such that
 \[
   p^{\lfloor i \eta \rfloor+c_5} N \subseteq \lambda_i(N) \subseteq
 L_i \cap N \subseteq p^{\min \{ \lfloor i \vartheta \rfloor, \lfloor
   i \eta \rfloor \} - b} N \qquad \text{for $i \in \N_0$.}
 \]
 Hence Proposition~\ref{pro:rigid} shows that, for suitable $c_6 \in
 \N$ with $c_6 \ge c_5$, the filtration series
 $L_i \cap N$, $i \in \N_0$, of $N$ is $c_6$-equivalent to a
 filtration series of the form $p^{n_i} N$, $i \in \N_0$, where
 \[
   \min \{ \lfloor i \vartheta \rfloor, \lfloor i \eta \rfloor \} - b
   \le n_i \le \lfloor i \eta \rfloor \qquad \text{for all $i \in \N_0$.}
 \]
 Recalling the notational
 `nullnummer'~\eqref{equ:streamline-w-equal-0}, we deduce that, for
 $c = c_4+c_6$, the lower $p$-series $L_i$, $i \in \N_0$, is
 $c$-equivalent to the filtration series in
 \eqref{equ:first-approxiamtion-of-Li}.%

 \smallskip
 
 At this stage we are ready to finish the proof in the case that
 $ \eta \le \vartheta$.  Indeed, in this situation
 \[
   w_{i,k} \in p^{\lfloor i \eta \rfloor - b} N \quad \text{and} \quad
   \lfloor i \eta \rfloor - b \le n_i \le \lfloor i \eta \rfloor
 \]
 for all relevant indices $i,k$ and hence a small adjustment of the
 series in~\eqref{equ:first-approxiamtion-of-Li} shows that
 $L_i$, $i \in \N_0$, is $(b+c)$-equivalent to the series
 \begin{equation*}
   \bigoplus_{k=1}^{\widetilde e} p^{\lfloor i \xi_k \rfloor}
   \Z_p x_k \;\oplus\; \bigoplus_{k = \widetilde e +1}^e
   \Z_p  p^{\lfloor i \xi_k \rfloor} x_k 
   \;\oplus\; p^{\lfloor i \eta \rfloor} N, \quad i \in \N_0.
 \end{equation*}
 We choose any $\Z_p$-basis $y_1, \ldots, y_{d-e}$ for $N$.  Then
 $L_i$, $i \in \N_0$, is equivalent to
 \[
   \bigoplus_{k=1}^e p^{\lfloor i \xi_k \rfloor} \Z_p x_k \;\oplus\;
   \bigoplus_{k=1}^{d-e} p^{\lfloor i \eta \rfloor} \Z_p y_k
  \]  
  and, if $k \in \{1, \ldots, \widetilde e\}$ is minimal
  subject to $\eta \le \xi_k$, this yields a stratification with frame
  \[
    (x_1, x_2, \ldots, x_{k-1}, y_1, y_2, \ldots, y_{d-e}, x_k,
    x_{k+1}, \ldots, x_e)
  \]
  and growth rate
  \[
    (\xi_1, \xi_2, \ldots, \xi_{k-1}, \underbrace{\eta, \eta, \ldots,
      \eta}_{\text{$d-e$ entries}}, \xi_k, \xi_{k+1}, \ldots, \xi_e) \in
    \big( [\nicefrac{1}{\max\{e,d-e\}},1] \cap \Q \big)^d.
  \]

  \smallskip

  Finally, suppose that $\vartheta < \eta$.  In this situation
  \[
    w_{i,k} = p^{\lfloor i \vartheta \rfloor} v_{i,k} \quad
    \text{with} \quad v_{i,k} \in p^{- b} N \qquad \text{and} \qquad
    \lfloor i \vartheta \rfloor - b \le n_i \le \lfloor i \eta
    \rfloor.
  \]
  Below we show the following:
  % Claim 3
  there exists $m \in \N_0$ such that for every choice of
  $\mathbf{v} = (v_{\widetilde e + 1}, \ldots,
  v_{\widehat e}) \in \big(p^{-b} N
  \big)^{\widehat e - \widetilde e}$, we have
  \begin{equation} \label{equ:uniform-lower-bound-in-N} p^m N
    \subseteq \tilde M + \sum_{k = \widetilde e
      +1}^{\widehat e} (x_k + v_k) . \Z_pG.
  \end{equation}
  From this we deduce that for $i \in \N_0$,
  \begin{multline*}
    L_i \supseteq p^c \Big( p^{\lfloor i \vartheta \rfloor} \tilde
    M + \bigoplus_{k = \widetilde e +1}^{{\widehat e}} p^{\lfloor i
      \vartheta \rfloor} \Z_p (x_k + v_{i,k} ) \Big) . \Z_pG \\
    \supseteq p^{c + \lfloor i \vartheta \rfloor} \Big( \tilde M +
    \sum_{k = \widetilde e +1}^{\widehat e} (x_k +
    v_{i,k}).\Z_pG \Big) \supseteq p^{c + m + \lfloor i
      \vartheta \rfloor} N,
  \end{multline*}
  hence
  $p^{n_i} N \supseteq p^c L_i \cap N \supseteq p^{2c + m + \lfloor i
    \vartheta \rfloor} N$, and we obtain
  \[
     w_{i,k} \in p^{\lfloor i \vartheta \rfloor - b} N \quad \text{and}
     \quad \lfloor i \vartheta \rfloor -b \le n_i \le \lfloor i \vartheta
    \rfloor + 2c + m
  \]
  for all relevant indices $i,k$.  Adjusting the series
  in~\eqref{equ:first-approxiamtion-of-Li}, we deduce that $L_i$,
  $i \in \N_0$, is $(b+3c+m)$-equivalent to the series
  \begin{equation*} 
    \bigoplus_{k=1}^{\widetilde e} p^{\lfloor i \xi_k \rfloor}
    \Z_p x_k \;\oplus\; \bigoplus_{k = \widetilde e +1}^e
    p^{\lfloor i \xi_k \rfloor} \Z_p x_k 
    \;\oplus\; p^{\lfloor i \vartheta \rfloor} N, \quad i \in \N_0.
  \end{equation*}
  We choose any $\Z_p$-basis $y_1, \ldots, y_{d-e}$ for $N$.  Then
  $L_i$, $i \in \N_0$, is equivalent to
  \[
    \bigoplus_{k=1}^e p^{\lfloor i \xi_k \rfloor} \Z_p x_k \;\oplus\;
    \bigoplus_{k=1}^{d-e}  p^{\lfloor i \vartheta \rfloor} \Z_p y_k
  \]
  and this yields a stratification with frame
  \[
    (x_1, x_2, \ldots, x_{\widetilde e}, y_1, y_2, \ldots, y_{d-e},
    x_{\widetilde e +1}, x_{\widetilde e + 2}, \ldots, x_e)
  \]
  and growth rate
  \[
    (\xi_1, \xi_2, \ldots, \xi_{\widetilde e}, \underbrace{\vartheta,
      \vartheta, \ldots, \vartheta}_{\text{$d-e$ entries}}, \xi_{\widetilde
      e +1}, \xi_{\widetilde e + 2}, \ldots, \xi_e) \in \big(
    [\nicefrac{1}{\max\{e,d-e\}},1] \cap \Q \big)^d.
  \]  

% Justification of Claim 3 
  It remains to justify \eqref{equ:uniform-lower-bound-in-N}, and for
  conciseness we put
  $\mathcal{V} = \big(p^{-b} N \big)^{\widehat e -
    \widetilde e}$.

For each
$\mathbf{v} = (v_{\widetilde e +1}, \ldots,
v_{\widehat e}) \in \mathcal{V}$, we think of
$\tilde T_\mathbf{v} = \sum_{k = \widetilde e
  +1}^{\widehat e} \Z_p (x_k + v_k)$ as a deformation
of~$T$ inside $M \oplus p^{-b}N \subseteq \Q_p \otimes_{\Z_p} L$.
Recall that, by our choice of~$\widetilde e$, the
$\Z_p$-lattice $\tilde M \oplus p^b \tilde T_\mathbf{v} \subseteq L$
is not $G$-invariant, whereas
$\tilde M \oplus p^b \tilde T_\mathbf{v} \oplus N = \tilde M \oplus
p^b T \oplus N$ is $G$-invariant by
Proposition~\ref{pro:stratification-restricts-to-small-xi}.  This
implies that
  \[
    \big( \tilde M + \tilde T_\mathbf{v}.\Z_pG \big) \cap N
    \ne \{0\}.
  \]
  Recall that $\Q_p \otimes_{\Z_p} N$ is a simple $\Q_pG$-module.
  Hence the non-zero $G$-invariant $\Z_p$-lattice
  $ \big( \tilde M + \tilde T_\mathbf{v}.\Z_pG \big) \cap N$ is open
  in $N$ of lower level $l_\mathbf{v} \in \N_0$, say; 
    compare with Proposition~\ref{pro:rigid} and the paragraph above it.

  Clearly, the function $\mathcal{V} \to \N_0$,
  $\mathbf{v} \mapsto l_\mathbf{v}$ is locally constant and, since the
  domain $\mathcal{V}$ is compact, the maximum
  $m = \max \{ l_\mathbf{v} \mid \mathbf{v} \in\mathcal{V} \}$ exists
  and has the property described in
  \eqref{equ:uniform-lower-bound-in-N}.
\end{proof}

\begin{remark}\label{rem:how-to-determine-xis}
  From the proof of Theorem~\ref{thm:stratification-exists} one can
  extract an inductive procedure for determining the growth rate
  $(\xi_1, \ldots, \xi_d)$ of a stratification for the lower
  $p$-series of the $\Z_p G$-module~$L$.

  Suppose that $d = \dim_{\Z_p}(L) \ge 1$, and let $N$ be a
  $\Z_p G$-submodule of $L$ such that $\Q_p \otimes_{\Z_p} N$ is a
  simple submodule of the $\Q_p G$-module $\Q_p \otimes_{\Z_p} L$.  By
  Proposition~\ref{pro:simple} the growth rate of a stratification for
  the lower $p$-series of $N$ is constant and thus equal to
  $(\eta, \ldots, \eta)$, say.  Replacing $L$ by a finite-index
  sublattice, as in the proof of the theorem, we may assume that $L/N$
  is again a $\Z_p$-lattice that is equipped with the induced
  $G$-action.  Suppose that $e = \dim_{\Z_p}(L/N) \ge 1$.  By
  induction, the growth rate of a stratification of the lower
  $p$-series for the $\Z_pG$-module $L/N$ can be computed and equals
  $(\xi_1,\ldots,\xi_e)$, say.  The next step is to identify
  $\widehat{e} \in \{1, \ldots, e\}$ and
  $\vartheta = \xi_{\widehat{e}}$, as described in the proof of the
  theorem.  This requires us to take into account how $N$ lies
  within~$L$, as shown by the example below.  The growth rate for $L$
  is obtained by supplementing the growth rate for $L/N$, i.e.\
  $(\xi_1, \ldots, \xi_e)$, in the correct position by $d-e$ entries
  equal to $\eta$, if $\eta \le \vartheta$, or by $d-e$ entries equal
  to $\vartheta$, if $\vartheta < \eta$.

  \smallskip
  
  We give a small explicit example to illustrate possible pitfalls.
  Let $\mathfrak{o}$ be the valuation ring of the local field
  $\Q_p(\pi)$, where $\pi^4 = p$, so that
  $\mathfrak{o} = \Z_p + \Z_p \pi + \Z_p \pi^2 + \Z_p \pi^3$ is a
  $\Z_p$-lattice of rank~$4$.  We consider the $16$-dimensional
  $\Z_p$-lattice
  $L = \bigoplus_{i=1}^4 \mathfrak{o}.x_i \cong \mathfrak{o}^4$,
  equipped with the natural right-action of the pro-$p$ group
  $G \le \mathsf{GL}_4(\mathfrak{o})$ generated by the four
  elements
  \[
    \begin{pmatrix}
      \boldsymbol{1+\pi}&0&0&0 \\
      0&\boldsymbol{1}&0&0 \\
      0&0&\boldsymbol{1}&0 \\
      0&0&0&\boldsymbol{1}
    \end{pmatrix},
    \begin{pmatrix}
      \boldsymbol{1}&\boldsymbol{1}&0&0 \\
      0&\boldsymbol{1}&0&0 \\
      0&0&\boldsymbol{1}&0 \\
      0&0&0&\boldsymbol{1}
    \end{pmatrix},
    \begin{pmatrix}
      \boldsymbol{1}&0&0&0 \\
      0&\boldsymbol{1}&0&0 \\
      0&0&\boldsymbol{1+\pi^2}&0 \\
      0&0&0&\boldsymbol{1}
    \end{pmatrix},
    \begin{pmatrix}
      \boldsymbol{1}&0&0&0 \\
      0&\boldsymbol{1}&0&0 \\
      0&0&\boldsymbol{1}&\boldsymbol{1} \\
      0&0&0&\boldsymbol{1}
    \end{pmatrix}.
  \]
  It is not difficult to work out the lower $p$-series of the
  $\Z_pG$-module $L$ and to deduce that is has a stratification of
  growth rate
  \[
    ( \underbrace{1/4,\ldots,1/4}_{\text{8 entries}},
    \underbrace{1/2,\ldots,1/2}_{\text{8 entries}}).
  \]
  Moreover, the $\Z_pG$-submodules $N$ of $L$ such that
  $\Q_p \otimes_{\Z_p} N$ is simple are precisely the $1$-dimensional
  $\Z_p$-submodules of
  $\{ y \in L \mid \forall g \in G: y^g = y \} = \mathfrak{o} x_2 +
  \mathfrak{o} x_4$.  For any such a module $N$ we can arrange that
  $N$ is a direct summand of the $\Z_p$-lattice~$L$, simply by passing
  to the isolator of $N$ in~$L$.  Furthermore, $N$
  has growth rate $(\eta)=(1)$, and $L/N$ has one of the following two
  possible growth rates
  \[
    (\underbrace{1/4,\ldots,1/4}_{\text{8
        entries}},\underbrace{1/2,\ldots.1/2}_{\text{7 entries}}),
    \qquad (\underbrace{1/4,\ldots,1/4}_{\text{7
        entries}},\underbrace{1/2,\ldots.1/2}_{\text{8 entries}}).
  \]
  Depending on how $N$ lies within $L$, the relevant value of
  $\vartheta$ is $1/2$ or $1/4$.
\end{remark}

Next we use $p$-adic Lie theory to derive
Theorem~\ref{thm:p-series-of-p-adic-group}.

\begin{proof}[Proof of Theorem~\ref{thm:p-series-of-p-adic-group}]
  Being $p$-adic analytic, the pro\nobreakdash-$p$ group $G$ has
  finite rank.  By \cite[Prop.~3.9]{DDMS99}, there exists $j \in \N$
  such that $U = P_j(G)$ is uniformly powerful of rank
  $\mathrm{rk}(U) = \dim(G) = d$ and such that $P_i(G)$ is powerfully
  embedded in $U$ for all $i \in \N$ with $i \ge j$.  We recall that
  there is an explicit isomorphism of categories translating between
  uniformly powerful pro\nobreakdash-$p$ groups and powerful
  $\Z_p$-Lie lattices; see~\cite[Sec.~4 and~9]{DDMS99} and
  \cite{Kl05}.  This means that the underlying set of the
  pro\nobreakdash-$p$ group $U$ can be equipped in a canonical way
  with the structure of a $\Z_p$-Lie lattice $L$ carrying the same
  topology and incorporating essentially all information about~$U$.
  In particular, the construction is such that exponentiation in $U$
  corresponds to scalar multiplication in $L$ and that the conjugation
  action of $G$ on $U$ translates to a continuous $\Z_p$-linear action
  of $G$ on~$L$.  In order to distinguish between the two parallel
  points of view, it is convenient to write $\underline{x}$ for
  $x \in U$ when it features as an element of the Lie lattice~$L$
  rather than as a group element.  Using this notation, the earlier
  statements can be phrased as follows: the (identity) map $U \to L$,
  $x \mapsto \underline{x}$ is a homeomorphism; we have
  $\underline{x^a} = a \underline{x}$ for all $x \in U$, $a \in \Z_p$;
  and conjugation, encapsulated in the canonical homomorphism
  $G \to \Inn(G)$, induces a continuous homomorphism
  \[
    G \to \Aut_{\Z_p}(L), \quad g \mapsto \alpha_g,
    \qquad \text{where $\underline{x} \alpha_g = \underline{x^g}$
      for all $x \in U$, $g \in G$.}
  \]
  Furthermore, for every $i \in \N$ with $i \ge j$, the
  powerfully embedded subgroup $P_i(G) \trianglelefteq_\mathrm{o} U$
  gives, on the Lie lattice side, a powerfully embedded Lie sublattice
  $L_{i^*} \trianglelefteq_\mathrm{o} L$, where we write $i^* = i-j$
  to ease notationally a recurrent constant shift by~$j$ in the
  index; see \cite[Thm.~1.4]{Kl05}.

  Clearly, $j^* = 0$ gives $L_{j^*} = L = \lambda_{j^*}(L)$, and we
  contend that, in fact, $L_{i^*} = \lambda_{i^*}(L)$ for all
  $i \ge j$.  Let $i \in \N$ with $i > j$.  Since $P_{i-1}(G)$
  is powerful, its Frattini subgroup equals $P_{i-1}(G)^p$ and
  corresponds to the Lie sublattice~$pL_{i^*-1}$.  Multiplicative
  cosets of $P_{i-1}(G)^p$ in $P_{i-1}(G)$ are the same as additive
  cosets of $pL_{i^*-1}$ in $L_{i^*-1}$, and multiplication in
  $P_{i-1}(G)/P_{i-1}(G)^p$ is the same operation as addition in
  $L_{i^*-1}/pL_{i^*-1}$; see \cite[Cor.~4.15]{DDMS99}.  Since
  \[
    \nicefrac{\displaystyle P_{i-1}(G)}{\displaystyle P_i(G)} \cong
    \nicefrac{\left( \displaystyle \nicefrac{\displaystyle
          P_{i-1}(G)}{\displaystyle P_{i-1}(G)^p} \right)}{\left(
        \displaystyle \nicefrac{\displaystyle P_i(G)}{\displaystyle
          P_{i-1}(G)^p} \right)}
  \]
  is the largest $G$-central quotient of the multiplicative group
  $P_{i-1}(G)/P_{i-1}(G)^p$, the corresponding term
  \[
    \nicefrac{\displaystyle L_{i^*-1}}{\displaystyle L_{i^*}} \cong
    \nicefrac{ \left( \displaystyle \nicefrac{ \displaystyle
          L_{i^*-1}}{\displaystyle p L_{i^*-1}} \right)}{\left(
        \displaystyle \nicefrac{\displaystyle L_{i^*}}{\displaystyle
          pL_{i^*-1}} \right)}
  \]
  is the largest $G$-central quotient of the additive group
  $L_{i^*-1}/pL_{i^*-1}$.  By induction, $L_{i^*-1}$ equals
  $\lambda_{i^*-1}(L)$ and thus its largest $G$-central quotient of
  exponent $p$ is $L_{i^*-1} / \lambda_{i^*}(L)$.  This implies that
  $L_{i^*}$ is indeed equal to $\lambda_{i^*}(L)$.

  Theorem~\ref{thm:stratification-exists} supplies us with a
  stratification of the lower $p$-series $L_{i^*}$, $i \in \N$
  with $i \ge j$, based on a frame
  $(\underline{x_1}, \ldots, \underline{x_d})$ and with a growth rate
  $(\xi_1, \ldots, \xi_d)$, say, where
  $\nicefrac{1}{d} \le \xi_1\le \ldots \le \xi_d \le 1$ are rational
  numbers.  This means in particular that
  $\underline{x_1}, \ldots, \underline{x_d}$ form a
  $\Z_p$-basis for $L$ and that there exists
  $c \in \N_0$ such that, for all $i \ge j$,
  \[
    \check{M}_{i^*} \subseteq L_{i^*} \subseteq \hat{M}_{i^*}
  \]
  for
  \[
    \check{M}_{i^*} = \bigoplus_{k=1}^d 
    p^{\lfloor i^* \xi_k \rfloor + c} \Z_p \underline{x_k} %
    \quad \text{and} \quad %
    \hat{M}_{i^*} = \bigoplus_{k=1}^d  p^{\lfloor
      i^* \xi_k \rfloor - c} \Z_p \underline{x_k},
  \]
  where once more we embed $L$ into
  $\Q_p \otimes_{\Z_p} L$ to resolve unambiguously
  multiplication by negative powers of~$p$.  By
  Corollary~\ref{cor:stratif-by-G-inv-lattices}, the lattices $\check{M}_{i^*}$
  and $\hat{M}_{i^*}$ are $G$-invariant for all $i \ge j$.
  
  We consider $i \in \N$ with $i \ge \max\{j+2cd,6cd\}$.  We see
  that
  \[
    [P_{i-2cd}(G),P_{i-2cd}(G)] \subseteq P_{2i-4cd}(G) \subseteq
    P_{i+2cd}(G).
  \]
  In conjunction with \cite[Cor.~3.5]{FGJ08}, the description of the
  Lie bracket in terms of the group multiplication, as described in
  \cite[Sec.~4]{DDMS99} and recorded in \cite[Equ.~(2.1)]{Kl05} yields
  the corresponding result on the Lie lattice side,
  \[
    [L_{i^*-2cd}, L_{i^*-2cd}]_\mathrm{Lie} \subseteq L_{i^*+2cd},
  \]
  where we write $[\cdot,\cdot]_\mathrm{Lie}$ for the Lie bracket to
  distinguish it from group commutators.  Similar to the simpler
  situation discussed above, the finite abelian group section
  $P_{i-2cd}(G)/P_{i+2cd}(G)$ of $G$ corresponds to a finite abelian
  Lie ring section $L_{i^*-2cd}/L_{i^*+2cd}$ of the $\Z_p$-Lie lattice
  $L$: the multiplicative cosets are the same as the additive cosets,
  multiplication and addition are the same and the Lie bracket is
  trivial; compare \cite[Thm.~4.5]{Go07}.

  Indeed, we already observed that the Lie bracket on $L_{i^*-2cd}$ is
  trivial modulo $L_{i^*+2cd}$, and the Lie ring
  $L_{i^*-2cd}/L_{i^*+2cd}$ is abelian.  In conjunction with
  \cite[Lem.~B.2]{Kl05}, the description of group multiplication via
  the Hausdorff series~$\Phi$, namely
  \[
    \underline{xy} = \Phi(\underline{x},\underline{y}) = \underline{x}
    + \underline{y} + \tfrac{1}{2}
    [\underline{x},\underline{y}]_\mathrm{Lie} +
    \sum_{n=3}^\infty u_n(\underline{x},\underline{y}) \qquad
    \text{for $x,y \in U$}
  \]
  as described in \cite{Kl05}, yields that multiplicative cosets of
  $P_{i+2cd}(G)$ in $P_{i-2cd}(G)$ and additive cosets of
  $L_{i^*+2cd}$ in $L_{i^*-2cd}$ agree with one another and that
  multiplication modulo $P_{i+2cd}(G)$ is the same operation as
  addition modulo $L_{i^*+2cd}$. Here the details are more tricky to
  work out and we indicate how to proceed when $p \ge 3$; the case
  $p=2$ is very similar, but would need notational refinements as in
  \cite{DDMS99,Kl05} which we want to forgo.  If
  $x,y \in P_{i-2cd}(G)$, then on the Lie lattice side we obtain
  \[
    u_2(\underline{x},\underline{y}) = \tfrac{1}{2}
    [\underline{x},\underline{y}]_\mathrm{Lie} \in
    [L_{i^*-2cd},L_{i^*-2cd}]_\mathrm{Lie} \subseteq L_{i^*+2cd}
  \]
  and, for $n \ge 3$,
  \begin{multline*}
    u_n(\underline{x},\underline{y}) \in p^{- \lfloor (n-1)/(p-1)
      \rfloor} \big[ \!\!\cdot\!\cdot\!\cdot\!\! \big[
    [L_{i^*-2cd},L_{i^*-2cd}]_\mathrm{Lie}, \underbrace{L_{i^*-2cd}
      ]_\mathrm{Lie},
      \ldots, L_{i^*-2cd}}_{\text{$n-2$ terms}} \big]_\mathrm{Lie} \\
    \subseteq p^{- \lfloor (n-1)/(p-1) \rfloor} p^{n-2} L_{i^*+ 2cd}
    \subseteq L_{i^*+2cd}.
  \end{multline*}
  This shows that all $x,y \in P_{i-2cd}(G)$ satisfy
  \begin{equation} \label{equ:mult-equiv-add}
    \underline{xy} = \Phi(\underline{x},\underline{y}) \equiv
    \underline{x} + \underline{y} \quad \text{modulo $L_{i^*+2cd}$.}
  \end{equation}
  In particular, the elements $x$ and $y$ are multiplicative inverses
  of one another modulo $P_{i+2cd}(G)$ if and only if $\underline{x}$
  and $\underline{y}$ are additive inverses of one another modulo
  $L_{i^*+2cd}$.  By considering inverses of inverses, we conclude
  that multiplicative cosets of $P_{i+2cd}(G)$ are the same as
  additive cosets of $L_{i^*+2cd}$.  Now \eqref{equ:mult-equiv-add}
  shows that multiplication modulo $P_{i+2cd}(G)$ is the same
  operation as addition modulo $L_{i^*+2cd}$: we have established the
  isomorphism of abelian groups
  \[
    \nicefrac{\displaystyle P_{i-2cd}(G)}{\displaystyle P_{i+2cd}(G)}
    \cong \nicefrac{\displaystyle L_{i^*-2cd}}{\displaystyle L_{i^*+2cd}} \qquad
    \text{via} \quad x P_{i+2cd}(G) \mapsto \underline{x} + L_{i^*+2cd}.
  \]

  Since $G$ acts unipotently on the finite $p$-group
  $P_i(G) / P_i(G)^{p^{2c}}$ of order $p^{2cd}$, we conclude that
  $P_{i+2cd}(G) \subseteq P_i(G)^{p^{2c}}$ and likewise
  $P_i(G) \subseteq P_{i-2cd}(G)^{p^{2c}}$.  On the Lie lattice side
  these inclusions translate into
  \[
    L_{i^*+2cd} \subseteq p^{2c} L_{i^*} \subseteq \check{M}_{i^*}
    \subseteq L_{i^*} \subseteq \hat{M}_{i^*} \subseteq p^{-2c} L_{i^*}
    \subseteq L_{i^*-2cd}.
  \]
  Recalling that $\check{M}_{i^*}$ and $\hat{M}_{i^*}$ are
  $G$-invariant and satisfy $p^{2c} \hat{M}_{i^*} = \check{M}_{i^*}$,
  we conclude that they yield open normal subgroups
  $\check{U}_i, \hat{U}_i \trianglelefteq_\mathrm{o} G$ such that
  \[
    \check{U}_i \subseteq P_i(G) \subseteq \hat{U}_i
  \]
  and such that the abelian factor group
  $\hat{U}_i / \check{U}_i \cong \hat{M}_{i^*}/\check{M}_{i^*}$ is
  homocyclic of exponent~$p^{2c}$.

  It remains to factorise $\check{U}_i, \hat{U}_i$ into products of
  procyclic subgroups.  We consider $\check{U}_i$; the proof for
  $\hat{U}_i$ works in the same way.  Since
  $\check{U}_i = \check{M}_{i^*}$ as sets, we certainly have
  $x_k^{\, p^{\lfloor i ^* \xi_k \rfloor+c}} \in \check{U}_i$ for
  $1 \le k \le d$.  This shows that
  \[
    \check{V}_i = \overline{\langle x_1^{p^{\lfloor i^* \xi_1\rfloor+c}}\rangle}
    \,\cdots\, \overline{\langle x_d^{p^{\lfloor i^*
          \xi_d\rfloor+c}} \rangle} \subseteq_\mathrm{c} U
  \]
  is a subset of $\check{U}_i$.  Furthermore, \cite[Prop.~3.7 and
  Sec.~4.2]{DDMS99} show that $\check{V}_i$ is the disjoint union of
  $p^{\sum_{k=1}^d (\lfloor i^*\xi_d\rfloor - \lfloor i^*
    \xi_k\rfloor) }$ cosets of the open normal subgroup
  \[
    \overline{\langle x_1^{p^{\lfloor i^* \xi_d\rfloor+c}}\rangle}
    \,\cdots\, \overline{\langle x_d^{p^{\lfloor i^* \xi_d\rfloor+c}}
      \rangle} = U^{p^{\lfloor i^*\xi_d\rfloor+c}}
    \trianglelefteq_\mathrm{o} U
  \]
  of index
  $\lvert U : U^{p^{\lfloor i^*\xi_d\rfloor+c}} \rvert = p^{(\lfloor
    i^*\xi_d\rfloor +c ) d}$.  This yields
  \[
    \mu_U(\check{V}_i) = p^{\sum_{k=1}^d (\lfloor i^*\xi_d\rfloor -
      \lfloor i^* \xi_k\rfloor) } \big\lvert U : U^{p^{\lfloor
        i^*\xi_d\rfloor+c}} \big\rvert^{-1} = p^{-cd-\sum_{k=1}^d \lfloor
      i^*\xi_k\rfloor},
  \]
  where $\mu_U$ denotes the normalised Haar measure on the group~$U$.
  We recall that the normalised Haar measures $\mu_U$ and $\mu_L$ on
  the compact space $U = L$ are the same; see \cite[Prop.~A.2]{Kl05}
  and \cite[Lem.~1.9.4]{Bo07}.  This gives
  \[
    \mu_U(\check{U}_i)= \mu_L(\check{M}_{i^*}) = \lvert L :
    \check{M}_{i^*} \rvert^{-1} = p^{-c d - \sum_{k=1}^d \lfloor i^*
    \xi_k\rfloor} = \mu_U(\check{V}_i),
  \]
  and hence $\check{U}_i = \check{V}_i$ as wanted.
\end{proof}

It does not take much extra effort to prove
Corollaries~\ref{cor:index-of-p-series-term} and
\ref{cor:xi-unique-for-groups}.

\begin{proof}[Proof of Corollary~\ref{cor:index-of-p-series-term}]
  We continue to use the notation established in the proof of
  Theorem~\ref{thm:p-series-of-p-adic-group}.  Let $i \in \N$
  with $i \ge \max\{j+2cd,6cd\}$.  In the special case $c=0$
  everything is easier and the following argument, which implicitly
  uses $c \ge 1$, can be skipped.  Recall that
  $A_i = \hat{U}_i / \check{U}_i$ is a homocyclic group of rank $d$
  and exponent~$p^{2c}$; it has a unique homocyclic subgroup of rank
  $d$ and exponent~$p^c$, namely $B_i = A_i^{\, p^c}$, whose pre-image
  in $\hat{U}_i$ is
  \[
    G_i = \overline{\langle x_1^{\, p^{\lfloor (i-j) \xi_1\rfloor}},
      \ldots, x_d^{\, p^{\lfloor (i-j) \xi_d\rfloor}} \rangle}
    \trianglelefteq_\mathrm{o} U.
  \]
  The group $P_i(G)$ contains $\check{U}_i$ and maps, modulo
  $\check{U}_i$, onto a subgroup $C_i$ of~$A_i$.  Thus
  $C_i \cong \prod_{k=1}^d C_{p^{m_{i,k}}}$ for suitable parameters
  $0 \le m_{i,1} \le \ldots \le m_{i,d} \le 2c$, and
  \[
    \bigl\lvert \log_p \lvert A_i:C_i \rvert - \log_p \lvert A_i:B_i
    \rvert \bigr\rvert \le \sum_{k=1}^d \lvert m_{i,k} - c \rvert \le
    cd.
  \]
  This yields
  \[
    \bigl\lvert \log_p \lvert G:P_i(G) \rvert - \log_p \lvert G:G_i
    \rvert \bigr\rvert \le cd,
  \]
  and clearly
  \[
    \log_p \lvert G : G_i \rvert = \log_p \lvert G : P_j(G) \rvert +
    \sum_{k=1}^d \lfloor (i-j) \xi_k \rfloor.
  \]
  Thus it suffices to observe that, with the notation
  $\sigma = \sum_{k=1}^d \xi_k$,
  \[
    \lfloor (i-j) \sigma\rfloor -d \le \sum_{k=1}^d \lfloor
    (i-j) \xi_k \rfloor \le \lfloor (i-j) \sigma\rfloor. \qedhere
  \]
\end{proof}

\begin{proof}[Proof of Corollary~\ref{cor:xi-unique-for-groups}]
  The uniqueness of the parameters $\xi_1, \ldots, \xi_d$ follows by
  looking at the elementary divisors of increasing sections, similar
  to the proof of Lemma~\ref{lem:same-growth-rates}.  Indeed, suppose
  that $j$, $U$, $\xi_1, \ldots, \xi_d$, $c$ and $\check{U}_i$ satisfy
  the assertions of Theorem~\ref{thm:p-series-of-p-adic-group}. Let
  $L$ be the $\Z_p$-Lie lattice associated to~$U$.  By
  \cite[Prop.~3.9]{DDMS99}, there is $j' \in \N$ with $j' \ge j$ such
  that, for all $i \in \N$ with $i \ge j'$, the subgroups
  $P_i(G), \check{U}_i \trianglelefteq_\mathrm{o} G$ are uniformly
  powerful and thus give rise to full $\Z_p$-sublattices
  $A_i, \check{M}_i \subseteq_\mathrm{o} \Q_p
    \otimes_{\Z_p} L$, by the Lie correspondence.

    For $i_1, i_2 \in \N$ with $j' \le i_1 \le i_2$ the finite abelian
    $p$-group $A_{i_1}/A_{i_2}$ decomposes as
    \[
      A_{i_1}/A_{i_2} \cong \Z_p / p^{m_{i_1,i_2,1}}\Z_p \,\oplus
      \ldots \oplus\, \Z_p / p^{m_{i_1,i_2,d}}\Z_p
    \]
    with parameters $0 \le m_{i_1,i_2,1} \le \ldots \le m_{i_1,i_2,d}$
    reflecting the elementary divisors.  As in the proof of
    Lemma~\ref{lem:same-growth-rates}, we conclude that 
    \[
      \xi_k = \lim_{\substack{j' \le i_1 \le i_2 \\ \text{s.t.}\ i_2 - i_1 \to
          \infty}} \frac{m_{i_1,i_2,k}}{i_2-i_1}  \qquad \text{for $1
        \le k \le d$.} \qedhere
    \]
\end{proof}

%%%%% 

\section{Hausdorff spectra with respect to the lower
  \texorpdfstring{$p$}{p}-series}\label{sec:hspec}

Our next task is to use the stratification results from Section~\ref{sec:stratifications} to pin down the Hausdorff dimensions of closed subgroups in $p$-adic analytic pro-$p$ groups  with respect to the lower $p$-series.

\begin{theorem} \label{thm:hspec-of-lattice}
  Let $L$ be a non-zero $\Z_p$-lattice of dimension
  $d = \dim_{\Z_p}(L)$.  Let $\mathcal{S}$ be a filtration
  series of $L$ which admits a stratification with growth rate
  $(\xi_1,\ldots,\xi_d)$.  Then the Hausdorff spectrum of the additive
  group $L$, with respect to $\mathcal{S}$, is
  \[
    \hspec^{\mathcal{S}}(L) = \left\{\frac{\varepsilon_1\xi_1 + \ldots
        + \varepsilon_d\xi_d}{\xi_1 + \ldots + \xi_d} \mid
      \varepsilon_1, \ldots, \varepsilon_d \in \{0,1\} \right\}.
  \]
  In particular, this shows that
  $\lvert \hspec^{\mathcal{S}}(L) \rvert \le 2^d$.

  Furthermore, every subgroup $H \subseteq_\mathrm{c} L$
  has strong Hausdorff dimension with respect to~$\mathcal{S}$.
\end{theorem}

\begin{proof}
  We write
 $\mathcal{S} \colon L = L_0 \supseteq L_1 \supseteq
    \ldots$ and choose a frame $(x_1, \ldots, x_d)$ for a
  stratification of~$\mathcal{S}$.  We also recall that
  $0 < \xi_1 \le \xi_2 \le \ldots \le \xi_d$.  By
  \cite[Lem.~2.2]{KlThZu19}, any two equivalent filtration series
  define the same Hausdorff dimension on closed subgroups of~$L$.
  Hence we may suppose that
  \[
    L_i = \bigoplus_{k=1}^d p^{\lfloor i \xi_k \rfloor}
    \Z_p \, x_k.
  \]
  
  Let $H \subseteq_\mathrm{c} L$ be a $\Z_p$-sublattice of
  dimension~$\dim_{\Z_p}(H) = r$, say.  Since $\Z_p$ is a principal
  ideal domain, $H$ has a $\Z_p$-basis $(y_1, \ldots, y_r)$ such
  that the parameters
  \[
    j(s) = \max \big( \{0\} \cup \{ k \mid 1\le k \le d \text{ and }
    a_{s,k} \neq 0 \} \big), \quad 1 \le s \le r,
  \]
  where we write $y_s = \sum_{k=1}^d a_{s,k} x_k$ with suitable
  coefficients $a_{s,k} \in \Z_p$, satisfy
  $j(1) < \ldots < j(r)$.  We contend that $H$ has strong Hausdorff
  dimension
  \[
    \hdim_L^{\mathcal{S}}(H) = \frac{\xi_{j(1)} + \ldots +
      \xi_{j(r)}}{\xi_1 + \ldots + \xi_d}.
  \]

  If $r=0$, the lattice $H$ equals $\{0\}$ and indeed
  $\hdim_L^\mathcal{S}(H) = 0$.  Now suppose that $r \ge 1$ and
  proceed by induction.  We write $H = H_1 \oplus H_2$, where
  $H_1 = \bigoplus_{s=1}^{r-1} \Z_p y_s$ and
  $H_2 = \Z_p y_r$.  By induction, we have
  \begin{equation}\label{equ:H1-limit}
    \lim_{i \to \infty} \frac{\log_p \lvert H_1 + L_i : L_i
      \rvert}{\log_p \lvert L :
      L_i \rvert} = \hdim_L^{\mathcal{S}}(H_1) = \frac{\xi_{j(1)} + \ldots +
      \xi_{j(r-1)}}{\xi_{1} + \ldots + \xi_{d}}.
  \end{equation}
  Now consider $H_2 = \Z_p y_r$ and recall that
  $y_r = \sum_{k=1}^{j(r)} a_{r,k} x_k$ with $a_{r,j(r)} \ne 0$.  Let
  $v_p \colon \Z_p \to \N_0 \cup \{\infty\}$ denote
  the $p$-adic valuation map, and put
  \[
    m = \min \{ v_p(a_{r,k}) \mid 1 \le k \le j(r) \text{ with } \xi_k
    = \xi_{j(r)} \} \in \N_0.
  \]
 For each $i \in \N$ the order of the finite cyclic
    group $(H_2 + L_i)/ L_i$ equals the order of $y_r$ modulo
    $L_i$.  Thus we see that, for all sufficiently large $i \in \N$,
  \[
    \log_p \lvert H_2 + L_i : L_i \rvert  = \max \{
      \lfloor i \xi_k \rfloor - v_p(a_{r,k}) \mid 1 \le k \le j(r) \} =
     \lfloor i \xi_{j(r)} \rfloor - m
  \]
  and thus
  \begin{equation}\label{equ:H2-limit}
    \lim_{i \to \infty}\frac{\log_p \lvert H_2 + L_i : L_i
      \rvert}{\log_p \lvert L :
      L_i \rvert} = \lim_{i \to \infty} \frac{\lfloor i\xi_{j(r)}\rfloor -
      m}{\lfloor i \xi_{1} \rfloor + \ldots + \lfloor i
      \xi_{d}\rfloor} = \frac{\xi_{j(r)}}{\xi_{1} + \ldots + \xi_{d}}.
  \end{equation}
  
  For each $i\in \N_0$, we see that
  \begin{multline*}
    \log_p \lvert H + L_i : L_i \rvert +\log_p \lvert (H_1 + L_i) \cap
    (H_2
    + L_i) : L_i \rvert \\
    = \log_p \lvert H_1 + L_i:L_i \rvert + \log_p \lvert H_2 + L_i:L_i
    \rvert,
  \end{multline*}
  and combining this observation with \eqref{equ:H1-limit} and
  \eqref{equ:H2-limit} we arrive at
  \[
    \frac{\xi_{j(1)}+\ldots+\xi_{j(r)}}{\xi_1+\ldots+\xi_d} -
    \hdim_L^\mathcal{S}(H) = \varlimsup_{i\rightarrow\infty}
    \frac{\log_p \lvert (H_1 + L_i) \cap (H_2 + L_i):L_i
      \rvert}{\log_p \lvert L : L_i \rvert}.
  \]
  It suffices to prove that the upper limit on the right-hand side
  equals $0$; in particular, this also shows that it is a proper
  limit.  Since $H_2 = \Z_p y_r$, we find that for each
  $i \in \N$ there exists $n_i \in \N_0$ such that
  \[
    (H_1 + L_i ) \cap (H_2 + L_i ) = p^{n_i} \Z_p \, y_r + L_i.
  \]
  Moreover,
  \[
    p^{n_i} y_r \in H_1 + L_i \subseteq \bigoplus_{k=1}^{j(r)-1}
    \Z_p x_k \;\oplus\; \bigoplus_{k=j(r)}^d p^{\lfloor i \xi_k
      \rfloor} \Z_p x_k
  \]
  shows that $n_i + v_p(a_{r,j(r)}) \ge \lfloor i\xi_{j(r)}\rfloor$
  for all $i \in \N_0$.  This implies that indeed
  \[
  \begin{split}
    \varlimsup_{i \to \infty} \frac{\log_p \lvert (H_1 + L_i) \cap
      (H_2 + L_i):L_i \rvert}{\log_p \lvert L : L_i \rvert} =
    \varlimsup_{i \to \infty} \frac{\log_p \lvert p^{n_i} \Z_p\, y_r
      + L_i:L_i\rvert}{\log_p \lvert L:L_i \rvert} \\
    \le \varlimsup_{i \to \infty} \frac{\lfloor i \xi_{j(r)} \rfloor -
      m - n_i}{\log_p \lvert L : L_i \rvert} \le \varlimsup_{i \to
      \infty} \frac{v_p(a_{r,j(r)}) - m}{\log_p \lvert L : L_i \rvert}
    = 0. \qedhere
    \end{split}
  \]
\end{proof}

We are now in position to prove Theorem~\ref{thm:spectrum-p-adic}.

\begin{proof}[Proof of Theorem \ref{thm:spectrum-p-adic}]
  As in the proof of Theorem~\ref{thm:p-series-of-p-adic-group} we
  find $j \in \N$ such that the subgroup
    $U = P_j(G) \subseteq_\mathrm{o} G$ is uniformly powerful and
  forms the starting point for the analysis carried out there.  In
  particular, there are matching filtration series
  \[
    \mathcal{S}_U \colon U_{i^*} = P_i(G), \; i \in \N_{\ge
      j}, \quad \text{and} \quad \mathcal{S}_L \colon L_{i^*} =
    \lambda_{i^*}(L), \; i \in \N_{\ge j},
  \]
  of the group $U$ and of the powerful $\Z_p$-Lie lattice $L$
  associated to~$U$, where we write $i^* = i-j$ as before.  

  Hausdorff dimensions of subgroups in $G$ with respect to the lower
  $p$-series $\mathcal{L}$ are closely linked to Hausdorff dimensions
  of subgroups in $U$ with respect to~$\mathcal{S}_U$.  Indeed, for
  every subgroup $H \subseteq_\mathrm{c} G$ and every
  open subgroup $W \subseteq_\mathrm{o} H$ with
  $W \subseteq U$,
  \[
    \hdim_G^\mathcal{L}(H) = \hdim_U^{\mathcal{S}_U}(W),
  \]
  and $H$ has strong Hausdorff dimension in $G$ if and only if $W$
  does in~$U$.

  Moreover, we claim that, if $W$ is chosen so that it is powerful
  (using \cite[Prop.~3.9]{DDMS99}) with corresponding Lie sublattice
 $M \subseteq_\mathrm{o} L$, then
  \begin{equation} \label{equ:hdim-W-M}
    \hdim_U^{\mathcal{S}_U}(W) = \hdim_L^{\mathcal{S}_L}(M),
  \end{equation}
  and again one of them is strong if and only if the other one is.
  The proof then concludes by the application of
  Theorems~\ref{thm:stratification-exists} and
  \ref{thm:hspec-of-lattice}.

  To see \eqref{equ:hdim-W-M}, let $i \in \N$ with $i \ge j$ and
  choose $m \in \N$ such that
  $U^{p^m} \subseteq U_{i^*}$, and accordingly
  $p^m L \subseteq L_{i^*}$.  By
  \cite[Cor.~4.15]{DDMS99}, multiplicative cosets of $U^{p^m}$ in~$U$
  correspond to additive cosets of $p^mL$ in~$L$.  Consequently,
  $U_{i^*}$ is the union of as many $U^{p^m}$-cosets as $L_{i^*}$ is
  the union of $p^mL$-cosets, hence
  \[
    \lvert U_{i^*} : U^{p^m} \rvert = \lvert L_{i^*} : p^m L \rvert.
  \]
  Since $W$ is powerful and $U_{i^*}$ is powerfully embedded in~$U$,
  the group $W U_{i^*}$ is powerful and corresponds to the Lie lattice
  $M+L_{i^*}$.  Again we obtain
  \[
    \lvert WU_{i^*} : U^{p^m} \rvert = \lvert M+L_{i^*} : p^m L
    \rvert.
  \]
  In combination this gives
  \begin{multline*}
    \frac{\log_p \lvert WU_{i^*} : U_{i^*} \rvert}{\log_p \lvert U :
      U_{i^*} \rvert} = \frac{\log_p \lvert WU_{i^*} : U^{p^m} \rvert
      - \log_p \lvert U_{i^*} : U^{p^m} \rvert}{\log_p \lvert U :
      U^{p^m} \rvert - \log_p \lvert U_{i^*} :
      U^{p^m} \rvert} \\
    = \frac{\log_p \lvert M + L_{i^*} : p^m L \rvert - \log_p \lvert
      L_{i^*} : p^m L \rvert}{\log_p \lvert L : p^mL \rvert - \log_p
      \lvert L_{i^*} : p^mL \rvert} = \frac{\log_p \lvert M + L_{i^*}
      : L_{i^*} \rvert}{\log_p \lvert L : L_{i^*} \rvert},
  \end{multline*}
  which shows that the Hausdorff dimension of $W$ in $U$ and the one
  of $M$ in $L$ are given by the lower limit of the same sequence.
\end{proof}

As indicated in the introduction, it is natural to look for sharp
bounds for the size of the Hausdorff spectrum $\hspec^\mathcal{L}(G)$
of a $p$-adic analytic group $G$, with respect to its lower
$p$-series, in terms of the dimension $\dim(G)$.   

\begin{example}\label{ex:max-value}
  For $m\in\N$, let $q_1, q_2, \ldots, q_m$ denote the first $m$
  consecutive primes, starting at $q_1=2$.  We consider the metabelian
  pro\nobreakdash-$p$ groups
  \[
    G_m = C \ltimes (A_1 \times \ldots \times A_m),
  \]
  where $C = \langle y \rangle \cong \Z_p$ is infinite procyclic and,
  for each $j \in\{1, \ldots, m\}$, the group
  $A_j= \langle x_{j,1}, \ldots, x_{j,q_j}\rangle \cong\Z_p^{\, q_j}$
  is torsion-free abelian of rank $q_j$ and normal in $G_m$, the
  action of $C$ on $A_j$ being determined by
  \[
    [x_{j,k}, y] =
    \begin{cases}
      x_{j,k+1} & \text{if $1 \le k \le q_j-1$,} \\
      x_{j,1}^{\, p} & \text{if $k = q_j$.}
    \end{cases}
  \]
  Clearly, the groups $G_m$ are $p$-adic analytic.  Below we determine
  asymptotically the size of their Hausdorff spectra, with respect to
  the lower $p$-series.  In particular, the estimate yields
  \begin{equation} \label{equ:Gm-asymp} \frac{\log_p \bigl\lvert
      \hspec^\mathcal{L}(G_m) \bigr\rvert}{\sqrt{\dim(G_m)}} \to \infty
    \qquad \text{as $m \to \infty$.}
  \end{equation}
  
  We fix $m \in \N$ and set $G = G_m$.  It is straightforward
  to check that $G$ has dimension
  \[
    \dim(G) = 1 + \sum_{j=1}^m q_j.
  \]
  Moreover, its lower $p$-series admits a stratification with growth
  rate
  \[
    (1,\, \nicefrac{1}{2},\nicefrac{1}{2}, \,
    \nicefrac{1}{3},\nicefrac{1}{3},\nicefrac{1}{3}, \, \ldots, \,
    \underbrace{\nicefrac{1}{q_m}, \ldots,
      \nicefrac{1}{q_m}}_{\text{$q_m$ entries}}).
  \]
  Theorem~\ref{thm:spectrum-p-adic} readily provides an upper bound
  for the size of the Hausdorff spectrum, namely
  \[
    \bigl\lvert \hspec^\mathcal{L}(G) \bigr\rvert \le 2 \prod_{j=1}^m (q_j+1).
  \]

  We arrive at a lower bound as follows.  For every
    subgroup $H \subseteq_\mathrm{c} G$ of
  the form $H = \prod_{j=1}^m B_j$ with
  $B_j \subseteq A_j$ and $\dim(B_j) = d_j$, say, a
  direct analysis, similar to the one in the proof of
  Theorem~\ref{thm:hspec-of-lattice}, yields
  \[
    \hdim^{\mathcal{L}}_G(H) = \frac{\nicefrac{d_1}{q_1} +
      \nicefrac{d_2}{q_2} + \ldots + \nicefrac{d_m}{q_m}}{m+1}.
  \]
  We observe that, upon restricting to $0 \le d_j< q_j$ for
  $j \in \{1,\ldots,m\}$, the formula yields
  \[
    \prod_{j=1}^m q_j
  \]
  distinct values in the Hausdorff spectrum of $G$ with respect to
  $\mathcal{L}$.

  The Prime Number Theorem provides the asymptotic equivalences
  \[
    \log \Big( 2 \prod_{j=1}^m (q_j+1) \Big) \sim \log \Big(
    \prod_{j=1}^m q_j \Big) \sim q_m \quad \text{and} \quad
    \sum_{j=1}^m q_j \sim \frac{q_m^{\, 2}}{2 \log q_m},
  \]
  as $m$ tends to infinity; compare with \cite[Cor.~1]{Sz80}
  (and \cite{Ax19} for a more careful approximation).
  These estimates describe the size of the Hausdorff spectrum
  asymptotically and imply \eqref{equ:Gm-asymp}.
\end{example}

%%%%%
  
%%%%%
%%%%%

\end{document}